\documentclass{ws-m3as}

\usepackage{graphicx}
\usepackage{latexsym}
\usepackage{amsmath}
\usepackage{enumerate}
\usepackage{amssymb}
\usepackage[mathscr]{eucal}
\usepackage{array}
\usepackage{todonotes}
\usepackage{showlabels}
\renewcommand{\theequation}{\arabic{section}.\arabic{equation}}

\newtheorem{thm}{Theorem}[section]

\newtheorem{lem}[thm]{Lemma}
\newtheorem{pro}[thm]{Proposition}

\newtheorem{ass}[thm]{Assumption}  %NEW!!!

\newtheorem{defn}[thm]{Definition}
\newtheorem{rem}[thm]{Remark}
\numberwithin{equation}{section}

\renewcommand{\d}{\mathrm{d}}

\newcommand{\R}{\mathbb{R}}

\newcommand{\Y}{\mathbb{Y}}
\newcommand{\V}{\mathbb{V}}
\newcommand{\Vj}{\mathcal{V}}

\newcommand {\E} { {\mathbb E} }
\newcommand{\Sph}{\mathbb{S}}
\newcommand{\eps}{\epsilon}

\newcommand{\e}{{\rm e}}
\newcommand{\M}{\mathcal M}

\newcommand{\dps}{\displaystyle}

\newcommand{\system}[1]{\left\{ \begin{alignedat}{2}#1\end{alignedat}\right.}

\newcommand{\emat}{\end{pmatrix}}
\newcommand{\abs}[1]{\left | #1\right |}

\newcommand{\pare}[1]{ \left(#1\right) }
\newcommand{\norm}[1]{\left\Vert#1\right\Vert}

\DeclareMathOperator{\Id}{{\rm Id}}

\renewcommand{\div}{ {\rm div}}

\newcommand{\disc}[1]{#1^{\delta t}}
\newcommand{\nl}{ {\rm r} }

\newcommand{\bigo}{\mathcal{O}}

\newcommand{\Max}{\mathcal M}
\newcommand{\bt}{{\bar{t}}}
\newcommand{\bs}{{\bar{s}}}
\newcommand{\one}{{\bf{1}}}
\newcommand{\taueps}{{\tau_{\eps}}}
\newcommand{\Feps}{{F_{\eps}}}

\begin{document}

\title{Individual-based models for bacterial chemotaxis in the diffusion aymptotics}
          %For each author, make a block with the following four macros:

\author{Mathias Rousset\thanks {SIMPAF, INRIA Lille - Nord Europe, Lille, France, (mathias.rousset@inria.fr).} 
{ \and} Giovanni Samaey\thanks {Department of Computer Science, K.U. Leuven,
Celestijnenlaan 200A, 3001 Leuven, Belgium, (giovanni.samaey@cs.kuleuven.be).}}
          %{Put the URL for your home page here if you have one}

          %Use \thanks statements for acknowledgements of grants and
          %support. They will appear below all the authors' addresses, so be
          %specific about which author is thanking whom:

          %\thanks{}

\pagestyle{myheadings} \markboth{Individual-based model of chemotaxis}{Rousset and Samaey}\maketitle

\begin{abstract}
We discuss velocity-jump models for chemotaxis of bacteria with an internal state that allows the velocity jump rate to depend on the memory of the chemoattractant concentration along their path of motion. Using probabilistic techniques, we provide a pathwise result that shows that the considered process converges to an advection-diffusion process in the (long-time) diffusion limit. We also (re-)prove using the same approach that the same limiting equation arises for a related, simpler process with direct sensing of the chemoattractant gradient.  Additionally, we propose a time discretization technique that retains these diffusion limits exactly, i.e., without error that depends on the time discretization.  In the companion paper \cite{variance},  these results are used to construct a coupling technique that allows numerical simulation of the process with internal state with asymptotic variance reduction, in the sense that the variance vanishes in the diffusion limit.     
\end{abstract}

\section{Introduction}

Generally, the motion of flagellated bacteria consists of a sequence of run phases, during which a bacterium moves in a straight line at constant speed. The bacterium changes direction in a tumble phase, which takes much less time  than the run phase and acts as a reorientation. To bias movement towards regions with high concentration of chemoattractant, the bacterium adjusts its turning rate to increase, resp.~decrease, the probability of tumbling when moving in an unfavorable, resp.~favorable, direction \cite{Alt:1980p8992,Stock:1999p8984}. Since many species are unable to sense chemoattractant gradients reliably due to their small size, this adjustment is often done via an intracellular mechanism that allows the bacterium to retain information on the history of the chemoattractant concentrations along its path \cite{Bren:2000p7499}. The resulting model can be formulated as a velocity-jump process, combined with an ordinary differential equation (ODE) that describes the evolution of an internal state that incorporates this memory effect \cite{ErbOthm04,Erban:2005p4247}. This model will be called the ``fine-scale'' or ``internal dynamics'' model in this paper.
%%%%%%%%%%%%
%% Je prefere ne pas avoir les equations dans l'intro, parce que je trouve que ce n'est pas trop comprehensible sans avoir lu la Section 2.
%%%%%%%%%%%%
%  The position of a single bacterium is denoted $t \mapsto X_t \in \R^d$, and the internal variable of the bacterium $t \mapsto Y_t\in \R^n$. $x \mapsto S(x) \in \R^n$ is a space function of the same dimension as the internal variable, and encoding all the necessary information on the chemoattractant concentrations. Then the instantaneous probability rate of occurrence of a tumble phase verifies:
% \[
% \lambda(X_t,Y_t) = \lambda_0 - b . (S(X_t)-Y_t) + \bigo( \abs{S(X_t)-Y_t}^k  ),
% \]
% where $k \geq 2$. In this context, the speed of a bacterium is the small parameter $\eps > 0$ (see Section~\ref{sec:asympt_time} for a rigorous dimension analysis). The internal variable then follows an ODE of the form:
% \[
%  \frac{ d Y_t}{ dt} = - \taueps ^{-1} (Y_t-S(X_t))+ \text{non-linear terms}, 
% \]
% where the linear part of the ODE is given by the matrix $\taueps ^{-1}$, which is of order $\taueps ^{-1} \sim \eps^{1-\delta}$ with $\delta > 1/k$.

The probability density of the velocity-jump process evolves according to a kinetic equation, in which, besides position and velocity, the internal variables appear as additional dimensions.  A direct deterministic simulation of the density distribution of all the variables of the model is therefore prohibitively expensive. Hence, it is of interest to study the relation of this model with simplified, coarse-level descriptions of the bacteria dynamics. 

In chemotaxis, a first coarse description is obtained by neglecting velocity and internal variables and considering the bacterial position density on large space and time scales. One then postulates an advection-diffusion equation for this bacterial position density, in which a chemotactic sensitivity coefficient incorporates the effect of chemoattractant concentrations gradients on the density fluxes. This assumption leads to the classical Keller--Segel equations (see \cite{KelSeg70}, and \cite{Horst1,Horst2} for numerous historical references). 
% In the present paper, we derive the stochastic differential equation (SDE) \eqref{eq:cprocess_hydro} that corresponds to asymptotic  advection-diffusion equation can be represented as an SDE (stochastic differential equation) which is the limit for small $\eps$ of the position of the bacterium $X^{\eps}_\bt:=X_{\bt/\eps^2}$ on diffusive time scales $\bt:=t  \eps^2$:
% \[
% \d X_{\bar{t}}^{0} =\pare{\frac{D A_0(X_{\bar{t}}^{0})}{\lambda_0} \d \bt + \pare{\frac{2D}{\lambda_0}}^{1/2} \d W_{\bar{t}}}.
% \]
% In the above, $ \bt \mapsto W_\bt$ is a standard Wiener process, $A_0 :=  b . \lim_{\eps \to 0}\frac{\taueps}{\lambda_0 \taueps + \Id } \nabla S(x)$, and $D$ if the covariance matrix of the random reorientation of velocity.
Several works have also considered the motion of a bacterium to be governed by a velocity-jump process corresponding to a kinetic description for the phase-space density \cite{ALT:1980p7984,Othmer:1988p7986,Patlak:1953p7738}. In these models, the velocity jump rate of the bacteria depends on the local chemoattractant gradient. In the present paper, the model corresponding to such bacteria will be called a ``coarse'' or ``direct gradient sensing'' model. Unlike the limiting Keller--Segel equation, see e.g.~\cite{Brenner:1999p9030,Herrero:1998p9032,Othmer:1997p9042}, the kinetic description does not exhibit finite-time blow-up (which is believed to be unphysical) under certain biologically relevant assumptions on the turning kernel \cite{Chalub:2004p7641}. These models also converge to a Keller--Segel equation in the appropriate drift-diffusion limit, see e.g.~\cite{Chalub:2004p7641,Hillen:2000p4751,Othmer:2002p4752}. In \cite{Erban:2005p4247}, the kinetic equation associated to a model with internal state (similar to the model considered here) has also be shown to formally converge to a Keller--Segel equation using moment closure assumptions and an appropriate (diffusion) scaling; a similar result using an infinite moment expansion has been obtained in \cite{XueOth09}. Then, the parameters of internal dynamics appear in the expression for the chemotactic sensitivity \cite{ErbOthm04,Erban:2005p4247}. Existence results when the model is coupled to a mean-field production of chemo-attractants are also available \cite{BourCal08}, and analysis of the long-time behavior is carried out in \cite{2010arXiv1006.0982F}.  Moreover, some recent studies have assessed the biological relevance of such a model through experimental validation of travelling pulses \cite{10.1371/journal.pcbi.1000890}. 

The present paper contributes to the analysis of velocity-jump processes for chemotaxis of bacteria with internal dynamics, and their numerical discretization, in the diffusive asymptotics. The main contributions are twofold~:
\begin{itemize}
	\item From an analysis point of view, we consider (fine-scale) velocity-jump processes with internal dynamics \eqref{eq:process_noscale}, with jump rates satisfying \eqref{eq:lin_rate} and internal dynamics satisfying the two main assumptions \eqref{eq:tau_eps} and \eqref{eq:tau_scales} as well as other technical assumptions contained in Section~\ref{sec:ass_ass}. We rigorously prove, using probabilistic arguments, the \emph{pathwise} convergence of the velocity-jump process described above to the stochastic differential equation (SDE)  \eqref{eq:cprocess_hydro} in the long time limit (see Proposition~\ref{pro:convproc}). The technical steps of the proof are given in the beginning of Section~\ref{sec:adv-diff}.		
		 We also recall a convergence result for the (coarse) process with direct gradient sensing, see Proposition~\ref{pro:convcoupl}. For an appropriate choice of the model parameters, both limits can be made identical. 
	\item From a numerical point of view, we introduce a time-discretized model (see Section~\ref{sec:discr}) for both the fine-scale and the coarse velocity-jump processes.  When discretizing the fine-scale process with internal dynamics, particular attention is paid to the time discretization errors due to the approximate solution of the ODE describing the evolution of the internal state. As a result, the asymptotic analysis of the time-continuous case still holds in an exact fashion, ensuring that the advection-diffusion limits are preserved after time discretization.
\end{itemize}

Due to the possibly high number of dimensions of the kinetic model with internal state, the evolution of the bacterial density away from the diffusive limit may rather be simulated using a stochastic particle method.  However, a direct stochastic particle-based simulation suffers from a large statistical variance, even in the diffusive asymptotic regime, raising the important issue of variance reduced simulation. In the companion paper \cite{variance}, we will use the results of the analysis presented here to construct a numerical scheme that couples the fine-scale model for bacteria with internal dynamics and the simpler, coarse model for bacteria with direct gradient sensing to obtain a variance-reduced simulation of the fine-scale process. We show that the variance reduction is \emph{asymptotic}, in the sense that the statistical variance scales as the error between the fine-scale and coarse description.

This paper is organized as follows. 
In Section~\ref{sec:particle}, we discuss the velocity-jump processes that will be considered, introducing both the fine-scale process with internal state and a simpler, coarse process in which the internal state has been replaced by direct gradient sensing. We discuss the relevant asymptotic regimes in Section~\ref{sec:asympt}, and analyze the limits of the considered processes in these regimes in Section~\ref{sec:adv-diff}. Section~\ref{sec:discr} discusses a discretization scheme for the velocity-jump processes that retains the diffusion limit of the continuum processes. Section~\ref{sec:concl} contains conclusions and a more detailed discussion of the relation with the companion paper \cite{variance}.

\section{Particle-based models for bacterial chemotaxis \label{sec:particle}}

\subsection{Model with internal state\label{sec:model-internal}}

We consider bacteria that are sensitive to the concentration of $m$ chemoattractants $\pare{\rho_i(x)}_{i=1}^m$, with $\rho_i(x) \geq 0$ for $x \in \R^d$. While we do not consider time dependence of chemoattractants via production or consumption  by the bacteria, a generalization to this situation is straightforward.  Bacteria move with a constant speed $v$ (run), and change direction at random instances in time (tumble), in an attempt to move towards regions with high chemoattractant concentrations.  As in \cite{ErbOthm04}, we describe this behavior by a velocity-jump process driven by some internal state $y \in \mathbb{Y}\subset\R^{n}$ of each individual bacterium. The internal state models the memory of the bacterium and is subject to an evolution mechanism attracted by a function  $\psi: \R^m \to \R^n$ of the chemoattractants concentrations,
\[
S(x):= \psi(\rho_1(x),\ldots,\rho_m(x)) \in \R^n,
\]
where $x$ is the present position of the bacterium. A typical choice is $2m = n$, $\psi(\rho_1, \ldots,\rho_m) = (\rho_1, \ldots,\rho_m,0, \ldots, 0)$, and the internal mechanism being such that the $m$ first internal variables are memorizing the concentrations of chemoattractants, while the $m$ last internal variables are computing the variations of concentrations along a bacterium trajectory (see Section~\ref{ssec:example}). $S$ and its two first derivatives are assumed to be smooth and bounded.

We use ($l_s$,$t_{\lambda}$) as a reference for length and time, where 
\begin{itemize}
  \item $l_s$ is the typical length of the chemoattractant concentration variations (which we assume similar for all the different species of chemoattractants).
   \item $t_{\lambda}$ is the typical time between two changes of the bacterium velocity direction (tumbling).
\end{itemize}
We also introduce the time $t_x$ that indicates the typical time on which a bacterium observes changes in chemoattractant concentration; the typical speed of a bacterium is thus $l_s/t_x$.  Anticipating the time-scale analysis in Section~\ref{sec:asympt}, we already introduce a small parameter 
\begin{equation}
  \label{eq:eps}
  \eps := \frac{t_\lambda}{t_x} \ll 1,
\end{equation}
indicating that we assume the typical time between two velocity changes to be much smaller than the typical time on which we observe macroscopic movement of the bacteria.  

The model is now presented directly in dimensionless form.  The evolution of each bacterium individual position is denoted as
\[
  t \mapsto X_t,
\]
with velocity
\[
  \frac{\d X_t}{\d t} = \eps V_t, \qquad V_t\in \mathbb{V}=\Sph^{d-1},
\]                             
where $\Sph^{d-1}$ represents the unit sphere in $\R^d$. The evolution of the internal state is denoted by
$t \mapsto Y_t$. The internal state adapts to the local chemoattractant concentration via an ordinary differential equation (ODE),
\begin{equation}\label{eq:internal}
  \frac{\d Y_t}{\d t}=  \Feps(Y_t,S(X_t)),
\end{equation}
which is required to have a unique fixed point $y^*=S(x^*)$ for every fixed value $x^* \in \R^d$. We also introduce the deviations from equilibrium $z=S(x)-y$.  The evolution of these deviations is denoted as 
\[
t \mapsto Z_t = S(X_t) -Y_t.
\]
The velocity of each bacterium is switched at random jump times
 $(T_n)_{n \geq 1}$
that are generated via a Poisson process with a time dependent rate given by $\lambda(Z_t)$, where $z \mapsto \lambda(z)$ is a smooth function satisfying
\begin{equation}
  \label{eq:ratebound}
  0 < \lambda_{\rm min} \leq \lambda \pare{z} \leq \lambda_{\rm max} .
\end{equation}
We will denote its expansion for small $z$ as
\begin{equation}\label{eq:lin_rate}
\lambda(z) = \lambda_0 - b^T  z +  c_\lambda  \bigo \pare{\abs{z}^k } ,
\end{equation}
with $k \geq 2$,  $b \in \R ^ n $; the constant $c_\lambda$ is used to keep track of the nonlinearity of $\lambda$ in the analysis.

The new velocity $\Vj_{n}$ at time $T_n$ is generated at random according to a centered probability distribution
$\M (dv)$ 
with $\dps \int v \, \M (dv) =0 $, typically
\[
\M (dv) = \sigma_{\Sph^{d-1}}(dv),
\]
where $\sigma_{\Sph^{d-1}}$ is the uniform distribution on the unit sphere.

The resulting fine-scale stochastic evolution of a bacterium is then described by a left continuous with right limits (lcrl) process
 \[
 t \mapsto \pare{X_t, V_t, Y_t} ,
 \] 
which satisfies the following differential velocity-jump equation:
\begin{equation} \label{eq:process_noscale}
  \begin{cases}
    \dps \dfrac{\d X_t}{\d t} = \eps  V_t  \\[8pt]
    \dps \dfrac{\d Y_t}{\d t} =  \Feps(Y_t,S(X_t)) \\[8pt]
    \dps \int_{T_n}^{T_{n+1}} \lambda(Z_t) \d t = \theta_{n+1}, \qquad \text {with }
    \dps Z_t := S(X_t)-Y_t \\[8pt]
    \dps V_{t}  =  \Vj_{n} \quad \text{ for $t \in [T_{n},T_{n+1}]$ } , 
  \end{cases}
\end{equation}
with initial condition $X_0, V_0 \in \R^d$, $Y_0 \in \R^n$ and $T_0 = 0$. 
In \eqref{eq:process_noscale}, $\pare{\theta _n }_{n \geq 1}$ denote i.i.d. random variables with normalized exponential distribution, and $\pare{ \Vj_n }_{n \geq 1}$ denote i.i.d. random variables with distribution $ \mathcal M (dv)$.

\begin{rem}
In one spatial dimension, assuming that $ \mathcal M (dv) = \frac{1}{2}(\delta_{+1}+\delta_{-1})$, the system \eqref{eq:process_noscale} is equivalent to the following system:
\begin{equation} \label{eq:process1d}
\system{
& \frac{\d X_t}{\d t} = \eps V_t  &\\
& \frac{\d Y_t}{\d t} = \Feps(X_t, S(Y_t)) & \\
& \int_{T_n}^{T_{n+1}} \lambda(Z_t) \d t =2\theta_{n+1}, \\
& V_t = \Vj_{n} \quad \text{ for $t \in [T_{n},T_{n+1}]$ with $\Vj_{n}=-\Vj_{n-1}$},& 
}
\end{equation}
in the sense that both processes have the same probability path distribution. This can be checked on the Markov infinitesimal generator of probability transitions, see \eqref{e:kinetic} and discussion below. (Note that, in this process, the velocity is always reversed. This is compensated by a factor $2$ in the equation defining the jump times.) 
\end{rem}

\subsection{Example}\label{ssec:example}

For concreteness, we provide a specific example, adapted from \cite{ErbOthm04}. We consider $m=1$, i.e.~there is only one chemoattractant $\rho_1(x)$.  We first describe a cartoon dynamics exhibiting an excitation-adaptation behaviour.  The internal state is two-dimensional, i.e. $n=2$, and $y=(y_1,y_2)$ satisfies the following ODE:
\begin{equation}\label{e:exc-ad}
	\system{
	&\frac{\d y_1(t)}{\d t} = \frac{\rho_1(x)-y_1(t)}{t_a},\\
	& \frac{\d y_2(t)}{\d t} = \frac{-\rho_1(x)+y_1(t)-y_2(t)}{t_e},
	}
\end{equation}
in which $t_a$, resp.~$t_e$, represent an adaptation, resp.~excitation, time; and $y_1$, resp. $y_2$, represent the adapting, resp.~exciting, variable.
This model has a single fixed point $y^*=(y_1^*,y_2^*)=(\rho_1(x),0)=:S(x)$, and the deviation variables are given by $z=(z_1,z_2)=S(x)-y=(\rho_1(x)-y_1,-y_2)$. When the excitation time is much smaller than the adaptation time, $t_e \ll t_a$, the variable $y_1$ therefore adapts ``slowly'' to the environment and memorizes it, while $y_2$ computes ``faster'' the lag $y_1-\rho_1(x)$, giving the response to changing environments. In the asymptotics $t_e/t_a \to 0$, the internal dynamics reduce to a scalar equation
\begin{equation}\label{e:scalar-y}
	\frac{\d y_1(t)}{\d t} = \frac{\rho_1(x)-y_1}{t_a},
\end{equation}
while the exciting variable reduces to the difference $y_2 =  y_1 - \rho_1(x) =-z_1$ which is instantly learned by the bacteria.  In this limit, we may eliminate the second internal state variable, and write $y\equiv y_1$ and  $S(x):=\rho_1(x)$, and $z=S(x)-y$.

For the turning rate $z\mapsto \lambda(z)$, we choose a nonlinear strictly decreasing smooth function, depending on a scalar $\zeta$ that aims at measuring the difference between the chemoattractant concentration and the memory variable. When considering the two-dimensional internal dynamics, we choose $\zeta=z_2=-y_2$, while, in the scalar case $\zeta=z$. As a consequence, $\zeta > 0$, resp.~$\zeta < 0$, give the information to the bacterium that it is moving in a favorable, resp.~unfavorable, direction.  A specific choice of $\lambda(\zeta)$ is
\begin{equation}\label{eq:rate}
\lambda(\zeta) = 2\lambda_0 \pare{\frac{1}{2}-\frac{1}{\pi}\arctan \pare{\frac{\pi}{2\lambda_0}\beta \zeta}}, \qquad \beta >0.
\end{equation}
Note that this turning rate satisfies~
\[
\lambda(z) = \lambda_0 - \beta \zeta + \bigo \pare{\abs{\zeta}^3}.
\]

\subsection{Model with direct gradient sensing}\label{sec:modelgrad}

We now turn to a simplified model, in which the internal process (\ref{eq:internal}), and the corresponding state variables, are eliminated.  Instead, the turning rate depends directly on the chemoattractant gradient. The process with direct gradient sensing is a Markov process in position-velocity variables
\[
t \mapsto (X^c_t,V^c_t) ,
\]
which evolve according to the differential velocity-jump equations
\begin{equation} \label{eq:cprocess}
\system{
& \frac{\d X_t^c}{\d t} = \eps V_t^c  &\\
& \int_{T_n^c}^{T_{n+1}^c} \lambda^c_\eps(X_t^c , V_t^c) \d t = \theta_{n+1} & \\
& V_{t}^c  =  \Vj_n  \quad \text{ for $t \in [T_{n}^c,T^c_{n+1}]$ } , & 
}
\end{equation}
with initial condition $X_0, \Vj_0 \in \R^d$.  In (\ref{eq:cprocess}), $\pare{\theta _n }_{n \geq 1}$ denote i.i.d. random variables with normalized exponential distribution, and $\pare{ \Vj_n }_{n \geq 1}$ denote i.i.d. random variables with distribution $ \Max (dv)$. 
The turning rate of the process with direct gradient sensing is assumed to be of the form
\begin{equation}\label{eq:control_rate}
\lambda^c_\eps(x,v) := \lambda_0 - \eps\; A_\eps^T(x)  v  + \bigo\left(\eps^{2}\right),
\end{equation}
and to satisfy
\begin{equation}
  \label{eq:rateboundc}
  0 < \lambda_{\rm min} \leq \lambda \pare{x,v} \leq \lambda_{\rm max} .
\end{equation}
The vector field $A_\eps(x) \in \R^{d}$ is usually a linear combination of the columns of $\nabla S(x) \in \R^{d \times n}$, and, as a consequence on the smoothness assumption on $S(x)$, satisfies that, for $\eps\to0$,  both $A_\eps$ and its derivatives approach a limiting vector field $A_0$ in uniform norm.

The model \eqref{eq:control_rate} may describe a large bacterium that is able to directly sense chemoattractant gradients. In the case $n=1$, the turning rate (\ref{eq:control_rate}) is proportional to $ \nabla S(x)  v \in \R$, which can be interpreted as follows: the rate at which a bacterium will change its velocity direction depends on the alignment of the velocity with the gradient of the chemoattractant concentration $\nabla S(x)$, resulting in a transport towards areas with higher chemoattractant concentrations. 

\section{Asymptotic regimes and kinetic formulation\label{sec:asympt}}

\subsection{Notation}
Throughout the text, the Landau symbol $\bigo$ denotes a \emph{deterministic} and globally Lipschitz function (with $\bigo(0)=0$). Its precise value \emph{may vary} from line to line, and may depend on the parameters of the model.
However, its Lipschitz constant is uniform in $\eps$. In the same way, we will denote generically by $C > 0$ a deterministic constant that may depend on all the parameters of the model except $\eps$.

\subsection{Definition of time scales}\label{sec:asympt_time}

Recall that we have introduced two time scales, characterized by (i) the typical time $t_x$ on which a bacterium observes changes in chemoattractant concentration; and (ii) the typical time $t_{\lambda}$ between two changes of the bacterium velocity direction (tumbling).
We have assumed a large time-scale separation between these two time scales by introducing the small parameter $\eps = t_\lambda/t_x \ll 1$.

Let us now turn to the typical time $t_y$ associated with the dynamics of the internal state.  In principle, several typical time scales $(t_{1,y} \gg t_{2,y} \gg \ldots)$ may be required to describe the evolution of the internal state; $t_y:= t_{1,y}$ is assumed to be the largest one. 
The vector field driving the ODE~\eqref{eq:internal} of the internal state is thus of order
\begin{equation*}
  \Feps \sim \dfrac{t_\lambda}{t_y},
\end{equation*}
and we will assume that 
\begin{equation}
  \label{eq:F}
\Feps \sim \dfrac{t_\lambda}{t_y} = \eps^{1-\delta},
\end{equation}
with $\delta > 1/k$, where $k$ is defined by~\eqref{eq:lin_rate}. 
% This leads to a third, intermediate time scale. 

Consequently, we will consider three time scales~:
\begin{itemize}
  \item A fast time scale of order $\bigo(1)$ given by the rate of change of the velocity direction.
  \item A time scale of order $\bigo(1/\eps^{1-\delta})$ with $\delta > 1/k$ given by the internal state evolution; depending on $\delta$, this time scale can be very fast, fast or intermediate.
 \item A slow time scale $\bigo(1/\eps)$ given by the evolution of the chemoattractant concentration as seen from the bacteria.
\end{itemize}

\subsection{Asymptotic assumptions}\label{sec:ass_ass}

The loose assumption on the scale, equation~\eqref{eq:F} for the internal  dynamics~\eqref{eq:internal}, needs to be made precise by the following set of assumptions. First, we assume that the ODE~\eqref{eq:internal} driving the internal state is well approximated by a near-equilibrium evolution equation~:
\begin{ass}\label{a:1}
 We have
\begin{equation}\label{eq:tau_eps}
\Feps(y,s) =  - \taueps ^{-1} (y-s) + \eps^{1-\delta}c_F\bigo(\abs{s-y}^2),
\end{equation} 
where $\taueps  \in\R^{n\times n}$ is an invertible constant matrix. The constant $c_F$ is used to keep track of the dependence on the non-linearity of $F_\eps$ in the analysis.
\end{ass}
We now formulate the necessary assumptions on the scale of the matrix $\taueps$~:
\begin{ass}\label{a:2}
There is a constant $C >0$ such that for any $t \geq 0$, one has
\begin{equation}\label{eq:tau_scales}
  \norm{\e^{- t \taueps ^{-1} }} \leq C \e^{- t \eps^{1-\delta}/C}.
\end{equation}
\end{ass}
Assumption~\ref{a:2} is used to ensure exponential convergence of the linear ODE with time scale at least of order $\eps^{1-\delta}$. When $\taueps^{-1}$ is symmetric, Assumption~\ref{a:2} follows from the assumption that the eigenvalues of $\taueps^{-1}$ are strictly positive, with a lower bound of order $\bigo(\eps^{1-\delta})$.

\begin{ass}\label{a:3}
There is a constant $C >0$ such that
\[
 \sup_{t \geq 0} \norm{t\taueps ^{-1}\e^{- t \taueps ^{-1} } } \leq C  \eps^{-1}.
\]
\end{ass}
Assumption~\ref{a:3} is necessary for technical reasons in Lemma~\ref{lem:DTestim}. When $\taueps^{-1}$ is symmetric, then Assumption~\ref{a:3} follows from the assumption that the eigenvalues of $\taueps^{-1}$ are non-negative.

Finally, we assume that the solution of the ODE driving the internal state in \eqref{eq:process_noscale} satisfies the following long time behavior~:
\begin{ass}
   \label{ass:ODE}
Consider a process $t \mapsto S_t$ such that $ \sup_{t \in [0,+\infty] } \abs{  d S_t/dt }= \bigo( \eps )$, and assume $\abs{S_0 - Y_0} = \bigo(\eps^{\delta})$. Then the solution of
\[
\frac{ d Y_t}{ dt} = \Feps(Y_t,S_t),
\]
satisfies
\[
\sup_{t \in [0,+\infty] } \abs{ Y_t - S_t } = \bigo \pare{ \eps^{\delta}  }.
\]
\end{ass}
Assumption~\ref{ass:ODE} is motivated by the case of linear ODEs; indeed, in that case, Assumption~\ref{ass:ODE} is a consequence of Assumption~\ref{a:2}, as is seen from the following Lemma~: 
\begin{lem}
  Assume that
\[
\Feps(y,s) =  - \taueps ^{-1} (y-s) ,
\]
then Assumption~\ref{a:2} implies Assumption~\ref{ass:ODE}.
\end{lem}
\begin{proof} Denoting $Z_t=S_t - Y_t$ in Assumption~\ref{ass:ODE}, Duhamel's integration formula yields
\begin{equation*}
   Z_t  = \e^{- {t}{\taueps^{-1}}} Z_{0} + \int_{0}^{t} \e^{- \pare{t-t'}{\taueps^{-1}}} \frac{d S_{t'}}{d t'}  \d t',
\end{equation*}
so that by Assumption~\ref{a:2}, we have
\begin{align*}
  \abs{Z_t} &\leq \abs{Z_0} + C \eps  \int_{0}^t \e^{- \pare{t-t'}{\eps^{1-\delta}}/C }   \d t' . \\
&\leq \abs{Z_0} + C \eps \times  \eps^{\delta-1}\pare{1- \e^{- t {\eps^{1-\delta}}/C }    } ,
\end{align*}
and Assumption~\ref{ass:ODE} follows.
\end{proof}
In a similar fashion, Assumption~\ref{a:3} implies the following technical bound~:
\begin{lem} \label{lem:techbound} Consider a process $t \mapsto S_t$ such that $ \sup_{t \in [0,+\infty] } \abs{  d S_t/dt }= \bigo( \eps )$. Assume that $ \abs{S_0 - Y_0} = \bigo(1)$, and that Assumption~\ref{a:1}, Assumption~\ref{a:2}, Assumption~\ref{a:3} and Assumption~\ref{ass:ODE} hold. Then, we have
  \begin{equation}
    \label{eq:techbound}
    \sup_{t \in [0,+\infty] } \abs{ \taueps^{-1}\e^{- \taueps^{-1} t } (S_t - Y_t) } = \bigo \pare{ 1 } .
  \end{equation}
\end{lem}
\begin{proof}
Let us denote $A_t :=   \taueps^{-1}\e^{- \taueps^{-1} t } (S_t - Y_t) $, which satisfies
\[
\frac{d A_{t}}{d t} = - 2 \taueps^{-1} A_t + \taueps^{-1}\e^{- \taueps^{-1} t } \pare{\frac{d S_{t}}{d t}  -F_{\eps}(Y_{t},S_{t} ) + 
\taueps^{-1}(S_{t}-Y_{t}) } .
\] 
Here again, Duhamel integration yields
\[
A_t = \e^{- {t}{2\taueps^{-1}}} A_{0} + \int_{0}^{t} \e^{- \pare{t-t'}{2\taueps^{-1}}}  \taueps^{-1}\e^{- \taueps^{-1} t' }\pare{  \frac{d S_{t'}}{d t'}  -F_{\eps}(Y_{t'},S_{t'} ) + 
\taueps^{-1}(S_{t'}-Y_{t'}) } \d t'.
\]
From Assumption~\ref{a:1} and Assumption~\ref{ass:ODE}, $ \abs{ F_{\eps}(Y_{t'},S_{t'} ) -
\taueps^{-1}(S_{t'}-Y_{t'}) } = \bigo(\eps^{1+\delta}) \leq \bigo(\eps)$ so that
\begin{align*}
  \abs{A_t} &\leq{A_{0}} +\norm{\taueps^{-1} \e^{- \taueps^{-1} t} } \int_{0}^{t} \norm{\e^{- \pare{t-t'}{\taueps^{-1}}} } \bigo(\eps) \d t' \\
& \leq{A_{0}} + \norm{t \taueps^{-1} \e^{- \taueps^{-1} t} } \frac{1}{t \eps^{1-\delta}}( 1 - \e^{ - t \eps^{1-\delta}/C  }  ) \bigo(\eps)\\
& \leq \bigo (1) + \bigo(\eps^{-1}) \bigo(\eps) = \bigo(1),
\end{align*}
where we have used in the above Assumption~\ref{a:2} and Assumption~\ref{a:3}.
\end{proof}
In the remainder of the paper, we will assume that the solution of~\eqref{eq:process_noscale} satisfies
\[
\abs{Z_0} = \bigo(\eps^{\delta}), 
\]
as well as:
\[
\abs{\taueps^{-1} Z_0} = \bigo( 1 ), 
\]
which implies, according to Assumption~\ref{ass:ODE} and Lemma~\ref{lem:techbound}, that
\[
\sup_{t \in [0,+\infty]} \pare{ \abs{Z_t} } =  \bigo(\eps^{\delta}),
\]
as well as
\[
\sup_{t \in [0,+\infty]} \pare{ \abs{ \taueps^{-1} \e^{- \taueps^{-1} t} Z_t} } =  \bigo(1) .
\]

Note that all the assumptions of the present section hold in the following case: (i) the ODE is linear ($F_\eps$ is linear), (ii) the associated matrix $\tau_{\eps}$ is symmetric positive and satisfies
\[
\tau_{\eps}^{-1} \geq C \eps^{1-\delta},
\]
for some $\delta > 0$, (iii) the initial internal state is close to equilibrium in the sense that $\taueps^{-1} Z_0=\bigo(1)$. The assumptions of the present section can be seen as technical generalizations to non-linear ODEs with non-symmetric linear part.

\subsection{Kinetic formulations}

The probability distribution density of the process with internal dynamics at time $t$ with respect to the measure $\d x \, \Max ( \d v ) \, \d y$ is denoted as
$p(x,v,y,t)$, suppressing the dependence on $\eps$ for notational convenience, and evolves according to the Kolomogorov forward evolution equation (or master equation). In the present context, the latter is the following kinetic equation
\begin{equation} \label{e:kinetic}
\partial_t p + \eps v \cdot \nabla_x p + \div_y  \pare{\Feps(y,S(x)) p } = \lambda \pare{S(x)-y} \pare{   R(p)    - p   },
\end{equation}
where 
\[
R(p) := \int_{v \in \Sph^{d-1}} p(\cdot,v,\cdot) \, \Max(dv)
\]
is the operator integrating velocities with respect to $\Max$. Similarly, the distribution density of the process with direct gradient sensing at time $t$ with respect to the measure $\d x \, \Max ( \d v ) $ is denoted as
$p^c(x,v,t)$, and evolves according to
the kinetic/master equation
\begin{equation} \label{e:kinetic_control}
\partial_t p^c + \eps v \cdot \nabla_x p^c = \pare{   R(\lambda^c_\eps(x,v) p)    - \lambda^c_\eps(x,v) p   }.
\end{equation}
We refer to \cite{EthKur86} for the derivation of master equations associated to Markov jump processes, and to \cite{Hillen:2000p4751,ErbOthm04,XueOth09,Chalub:2004p7641,Othmer:2002p4752} for an asymptotic analysis using moment closures.

\section{Asymptotic analysis of individual-based models\label{sec:adv-diff}}

In this section, we investigate in more detail the asymptotic behavior of the process with internal dynamics and the simplified process with direct gradient sensing.   We rigorously prove pathwise convergence to an advection-diffusion limit of both the process with direct gradient sensing (Section \ref{sec:limit-control}) and the process with internal dynamics (Section \ref{sec:limit-internal}).  Choosing the model parameters appropriately, both advection-diffusion limits can be made identical.
		The two proofs are based on an asymptotic expansion of the jump times (the time between two tumble phases), and a comparison with a simpler random walk. They can be split into $4$ steps.
\begin{enumerate}[Step (i)]
             \item Consider the evolution of the bacterium indexed by the number of tumble phases (jump times), as well as a proper Markovian random walk approximating the former. Show tightness of the coupled sequence (in fact their time continuous version obtained by linear interpolation).
             \item Show, using a Gronwall argument, that the two processes have the same advection-diffusion limit. This step relies on an asympotic expansion of the jump times in $\eps$.
               \item Compute the advection-diffusion limit of the Markovian random walk.
\item Use a random time change to transfer the result to the bacteria position process itself. 
          \end{enumerate}

\subsection{Diffusion limit of the process with direct gradient sensing\label{sec:limit-control}}

We first derive a pathwise advection-diffusion limit of the process with direct gradient sensing (\ref{eq:cprocess}) when $\eps \to 0$ using classical probabilistic arguments. It should be noted that this diffusive limit as already been extensively studied in the context of bacterial chemotaxis, see, e.g.,\cite{Hillen:2000p4751} for justifications of Hilbert expansions at the PDE level, and \cite{Stro74} where the diffusive limit of some more general class of stochastic processes is suggested. However, we give the probabilistic arguments {\it in extenso} since they will be of interest for the analysis of both the model with internal state, and of the coupling in the variance-reduced numerical scheme, discussed in the companion paper \cite{variance}. Moreover, they do not seem to appear as such in the literature.

Diffusive times are denoted by
\[
\bt:=t  \eps^2,
\]
and the process with direct gradient sensing considered on diffusive time scales by
\[
X^{c,\eps}_\bt:=X^c_{\bt/\eps^2}.
\]
This section shows that this process converges for $\eps \to 0$ towards an advection-diffusion process, satisfying the stochastic differential equation (SDE)
\begin{equation}\label{eq:cprocess_hydro}
\d X_{\bar{t}}^{c,0} =\pare{\frac{D A_0(X_{\bar{t}}^{c,0})}{\lambda_0} \d \bt + \pare{\frac{2D}{\lambda_0}}^{1/2} \d W_{\bar{t}}},
\end{equation}
where $\bt \mapsto  W_\bt$ is a standard Brownian motion, and the diffusion matrix is given by the covariance of the Maxwellian distribution:
\begin{equation}\label{eq:maxw}
D =  \int_{\Sph^{d-1}} v \otimes v \, \Max(dv) \in \R^{d \times d}.
\end{equation}
In particular, this result implies at the level of the Kolomogorov/master evolution equation, that the position bacterial density \begin{equation}\label{eq:cdensity}
	n^{c,\eps}(x,\bt):= n^c(x,\bt/\eps^2):=\int_{\V}p^c(x,v,\bt/\eps^2) \,  \Max(\d v) 
\end{equation}
 satisfies the advection-diffusion equation \begin{equation}\label{eq:cdens_hydro}
\partial_\bt n^{c,0} = \frac{1}{\lambda_0}\div_x \pare{ D \nabla_x n^{c,0} - D A_0(x) n^{c,0}}
\end{equation}
on diffusive time scales as $\epsilon \to 0$. 

The proof relies on a perturbation analysis of the jump time differences, which is given by the following lemma~:
\begin{lem}\label{lem:dt-control}
The difference between two jump times of the process with direct gradient sensing satisfies
\begin{equation}\label{eq:deltaTc}
\Delta T_{n+1}^{c} :=  T_{n+1}^{c}- T_{n}^{c} = \frac{\theta_{n+1}}{\lambda_0} + \eps \frac{\theta_{n+1}}{\lambda_0^2} A_\eps^T(X_{T^c_n}^c) \Vj_n + \theta_{n+1} \bigo (\eps^2  ).
\end{equation}
\end{lem}
\begin{proof}
Using \eqref{eq:control_rate}, we get
\begin{align*}
 \lambda^c_\eps(X_t^c,V_t^c) & = \lambda_0 - \eps   A_\eps^T(X_t^c)  \Vj_n  + \bigo (\eps^2 ) \\
& = \lambda_0 - \eps    A_\eps^T (X^c_{T^c_n}) \Vj_n + (\norm{\nabla A}_\infty +1)\bigo (\eps^2 ) \\
& = \lambda_0 - \eps    A_\eps^T (X^c_{T^c_n}) \Vj_n + \bigo (\eps^2 ),
\end{align*}
where we have used differentiability of $A_\eps(x)$, and the property $\|X_t^c-X^c_{T^c_n}\|\le \eps(t-T^c_n)$.
Integrating on the time interval $[T^c_n,T^c_{n+1}]$ yields
\[
\theta_{n+1} = \Delta T_{n+1}^{c} \pare{ \lambda_0 - \eps    A_\eps^T(X^c_{T^c_n})  \Vj_n + \bigo (\eps^2 )  },
\]
so that
\begin{align*}
  \Delta T_{n+1}^{c} &= \frac{\theta_{n+1} }{\lambda_0 - \eps    A_\eps^T(X^c_{T^c_n})  \Vj_n + \bigo (\eps^2 ) } \\
&=  \frac{\theta_{n+1} }{\lambda_0} \pare{1 + \eps \frac{1}{\lambda_0} A_\eps^T(X^c_{T^c_n})  \Vj_n}    + \theta_{n+1} \bigo (\eps^2 ) .
\end{align*}
\end{proof}

We also need the following technical lemma~:

\begin{lem}[Random time change]\label{l:techrandt}
 Let $\pare{ \bar t \mapsto \alpha_\eps(\bar t) } _{\eps \geq 0}$ a sequence of continuous, strictly increasing random time changes in $\R^+$, which, for $\eps\to 0$, converges in distribution (on any time interval for the uniform convergence topology) towards the identity function $\bar{t} \mapsto \bar{t}$. Let $\pare{ \bar t \to X_\eps(\bar t) } _{\eps \geq 0}$ a sequence of continuous random processes, such that 
\[
  \bar t \to X^\eps_{ \alpha_\eps(\bar t) }
\]
converges in distribution when $\eps \to 0$ (on any time interval for the uniform convergence topology) towards a continuous process $\bar t \to X^0_{ \bar t }$. Then the sequence $\pare{ \bar t \to X^\eps_{\bar t } } _{\eps \geq 0}$ also converges in distribution when $\eps \to 0$ (on any time interval for the uniform convergence topology) towards $\bar t \to X^0_{\bar t}$.
\end{lem}
\begin{proof}
Since by assumption $(\bt \mapsto \alpha^{\eps}(\bt))_{\eps \geq 0}$ converges to a deterministic limit, the coupled process $\bt \mapsto (X^{\eps}_{\alpha_{\eps}(\bt)}, \alpha_{\eps}(\bt) )$ also converges in distribution to $\bt \mapsto (X^{0}_{\bt}, \bt)$, and we can consider a Skorokhod embedding (a new probability space) associated with a $\eps$-sequence where this convergence holds almost surely. Then we can write ($\wedge$ denotes the minimum, and $\one$ the identity function):
\begin{align*}
    &\sup_{\bt \in [0,T]} \abs{ X^{\eps}_{\bt} - X^{0}_{\bt}} \wedge 1  \\
&\quad \leq \sup_{\bt \in [0,T]} \abs{ X^{\eps}_{\bt} - X^{0}_{\alpha^{-1}_{\eps}(\bt)} } \wedge 1 + \sup_{\bt \in [0,T]} \abs{ X^{0}_{\bt} - X^{0}_{\alpha^{-1}_{\eps}(\bt)}  } \wedge 1  , \\
& \quad \leq \sup_{\bs \in [0,\alpha^{-1}_{\eps}(T)]} \abs{ X^{\eps}_{\alpha^{\eps}(\bs)} - X^{0}_{\bs} } \wedge 1 + \sup_{\bs \in [0,\alpha^{-1}_{\eps}(T)]} \abs{ X^{0}_{\alpha_{\eps}(\bs)} - X^{0}_{\bs}  } \wedge 1 \\
& \quad \leq \sup_{\bs \in [0, 2T]} \abs{ X^{\eps}_{\alpha^{\eps}(\bs)} - X^{0}_{\bs} } \wedge 1 + \sup_{\bs \in [0, 2T]} \abs{ X^{0}_{\alpha_{\eps}(\bs)} - X^{0}_{\bs}  } \wedge 1 + 2 \times \one_{ \alpha^{\eps}(2T) \leq T}
\end{align*}
and the right hand side converges almost surely to $0$ using the (uniform) continuity of $\bt \mapsto X^{0}_{\bt}$.
\end{proof}

Then the following result holds~:
\begin{pro}\label{pro:convcoupl} Consider the gradient sensing process described in Section~\ref{sec:modelgrad}.
The process $\bt \mapsto X_{\bar{t}}^{c,\eps} = X_{t/\eps^2} ^c$ converges in distribution (with respect to the topology of uniform convergence) towards $\bt \mapsto X_{\bar{t}}^{c,0}$, solution of the SDE (\ref{eq:cprocess_hydro}). As a consequence, the density $n^{c,\eps}$ of the process with direct gradient sensing defined in \eqref{eq:cdensity} converges weakly towards $n^{c,0}$, solution of (\ref{eq:cdens_hydro}). 
\end{pro}
\begin{proof}
	The proof requires the definition of some auxiliary processes. First, we perform a random time change on the process of positions of the bacterium, so that the jump times become deterministic, and will occur on regular time instances separated by $\eps^2/\lambda_0$ on the diffusive time scale; the randomness is then transferred to the jump size. The auxiliary process is thus defined as
\begin{equation}
  \label{eq:sum_X_c}
  \tilde{X}^{c,\eps}_{\bt} 
= X^c_ 0 + \sum_{n=0}^{\lfloor  \lambda_0 \bt/\eps^2\rfloor - 1} \eps \Delta T^c_{n+1} \Vj_{n} + \underbrace{\pare{ \lambda_0 \bt/\eps^2 -\lfloor  \lambda_0 \bt/\eps^2\rfloor}\eps \Delta T^c_{\lfloor  \lambda_0 \bt/\eps^2\rfloor+1} \Vj_{\lfloor  \lambda_0 \bt/\eps^2\rfloor}}_{r_1}.
\end{equation} 
The latter will be compared with a classical drifted random walk defined by the Markov chain
\begin{equation}
  \label{eq:Markov}
  \begin{cases}
    \Xi^c_0 &= X^c_ 0 \\
\Xi^c_{n+1} &= \Xi^c_{n} + \eps \dfrac{\theta_{n+1} \Vj_n}{\lambda_0} + \eps^2 \dfrac{\theta_{n+1}\Vj_n}{\lambda_0^2} A^T(\Xi^c_{n}) \Vj_n,\\
&= \Xi^c_{n} + \eps \left(\dfrac{\theta_{n+1} }{\lambda_0} \pare{1 + \eps \dfrac{1}{\lambda_0} A^T(\Xi^c_{n})  \Vj_n}\right)\Vj_n,  
  \end{cases}
\end{equation}
to which we can associate a continuous process defined on a similar diffusive time scale by interpolation,
\begin{equation}
  \label{eq:sum_Xi_c}
  \tilde{\Xi}^{c,\eps}_{\bt} 
= X^c_ 0 + \pare{\sum_{n=0}^{\lfloor  \lambda_0 \bt/\eps^2\rfloor - 1} \pare{\Xi^c_{n+1}-\Xi^c_n}}+ \underbrace{\pare{ \lambda_0 \bt/\eps^2 -\lfloor  \lambda_0 \bt/\eps^2\rfloor}\pare{ \Xi^c_{\lfloor  \lambda_0 \bt/\eps^2\rfloor + 1}-\Xi^c_{\lfloor  \lambda_0 \bt/\eps^2\rfloor } }}_{r_2}.
\end{equation}

Now the computation of the diffusive limit of the \emph{sequence} of processes $\left( \bt \mapsto  \tilde{X}^{c,\eps}_{\bt}\right)_\eps$ for $\eps\to 0$ relies on four steps. \\
          \begin{enumerate}[Step (i)]
             \item Show tightness of the coupled sequence $\left(\bt \mapsto (\tilde{X}^{c,\eps}_{\bt},\tilde{\Xi}^{\eps}_{\bt})\right)_{\eps \geq 0}$.  To this end, decompose the leading terms in these processes into a ``martingale'' term and a drift term.
             \item Show, using a Gronwall argument, that the process describing the bacterial position at jump times $\tilde{X}^{c,\eps}_{\bt}$ and the classical Markovian random walk $\tilde{\Xi}^{\eps}_{\bt}$ have the same advection-diffusion limit.
               \item Compute the advection-diffusion limit of the Markovian random walk  $\tilde{\Xi}^{\eps}_{\bt}$ (and hence of the process at jump times $\tilde{X}^{c,\eps}_{\bt}$ ).  
\item Use a random time change to transfer the result on bacteria position process ${X}^{c,\eps}_{\bt}$ itself. 
          \end{enumerate}

Step (i). Tightness of the coupled sequence $\bt \mapsto (\tilde{X}^{c,\eps}_{\bt},\tilde{\Xi}^{\eps}_{\bt})$. \\
Let us recall that according to Kolmogorov's criterion, a sequence of continuous processes $\pare{ \bt \mapsto \omega_{\bt}^\eps}_{\eps \geq 0}$ is tight in the space $\mathcal C([0,\bar{T}], \R^d)$ of continuous processes endowed with uniform convergence if there exists a constant $C$ independent of $\eps$ such that
\begin{equation}
\label{eq:Kolmogorov}
  \E \abs{\omega_{\bt}^\eps - \omega_{\bs}^\eps}^4 \leq C\pare{\bt - \bs}^2,
\end{equation}
for any $0 \leq \bs \leq \bt \leq \bar{T}$. In our case, checking~\eqref{eq:Kolmogorov} for the coupled process $\bt \mapsto (\tilde{X}^{c,\eps}_{\bt},\tilde{\Xi}^{c,\eps}_{\bt})$ is highly standard (see \cite{StrVarSri06}).
Remark that it is sufficient to check the criteria for each of the two separately. We recall the computation for the process $\bt \mapsto \tilde{\Xi}^{c,\eps}_{\bt}$, the case of $\bt \mapsto \tilde{X}^{c,\eps}_{\bt}$ being similar using~\eqref{eq:deltaTc}. First, one can check~\eqref{eq:Kolmogorov} when $\bt - \bs \leq \lambda_0 \eps^2$. Indeed, in the latter case, only the rest $r_2$ (see equation~\eqref{eq:sum_Xi_c})  is involved and
\[
\E \abs{\tilde{\Xi}^{c,\eps}_{\bt} -\tilde{\Xi}^{c,\eps}_{\bs}  }^4 \leq C \frac{\pare{\bt - \bs}^4}{\eps^8}\eps^4 \leq C \pare{\bt - \bs}^2,
\]
where we have used that  
\begin{equation}\label{eq:est-eps-here}
\abs{ \Xi^c_{\lfloor  \lambda_0 \bt/\eps^2\rfloor + 1}-\Xi^c_{\lfloor  \lambda_0 \bt/\eps^2\rfloor } } \leq C\eps.
\end{equation}
Assume now that $\bt - \bs > \lambda_0 \eps^2$. The terms coming from $r_2$ can be easily controlled, since, again using~\eqref{eq:est-eps-here},
\[
\E \abs{r_2}^4 \leq C \eps^4 \leq C \pare{\bt - \bs}^2.
\]
The terms of order~$\eps$ in the definition~\eqref{eq:Markov} (defining a martingale) can be controlled by expanding the sum and eliminating the terms with null average. It yields:
\begin{align*}
  \E \abs{ \sum_{n=\lfloor  \lambda_0 \bs/\eps^2\rfloor}^{\lfloor  \lambda_0 \bt/\eps^2\rfloor - 1} \eps \frac{\theta_{n+1} \Vj_n}{\lambda_0}  }^4  &\leq C \eps^4 \pare{\lfloor  \lambda_0 \bt/\eps^2\rfloor - \lfloor  \lambda_0 \bs/\eps^2\rfloor}^2 \\
&\leq C \eps^4 \pare{\bt/\eps^2 -\bs/\eps^2 }^2 + C \eps^4   \\
&\leq  C \pare{\bt - \bs}^4 \leq C \pare{\bt - \bs}^2.
\end{align*}
In the last line, we use the fact that $\pare{\bt - \bs}^2 \leq 4 \bar{T}^2 \leq C$. Finally the terms of order~$\eps^2$ in the definition~\eqref{eq:Markov} (defining a drift) can be controlled as
\begin{align*}
  \E \abs{ \sum_{n=\lfloor  \lambda_0 \bs/\eps^2\rfloor}^{\lfloor  \lambda_0 \bt/\eps^2\rfloor - 1} \eps^2 \frac{\theta_{n+1}\Vj_n}{\lambda_0^2} A_\eps^T(\Xi^c_{n}) \Vj_n }^4 & \leq  C \eps^8 \pare{\lfloor  \lambda_0 \bt/\eps^2\rfloor - \lfloor  \lambda_0 \bs/\eps^2\rfloor}^4 \\
& \leq C \eps^8 \pare{\bt/\eps^2 -\bs/\eps^2 }^4 + C\eps^8   \\
& \leq C \pare{\bt - \bs}^2,
\end{align*}
where we again use the fact that $\pare{\bt - \bs}^2 \leq 4 \bar{T}^2 \leq C$. 
We conclude that \[
\E \abs{\tilde{\Xi}^{c,\eps}_{\bt} -\tilde{\Xi}^{c,\eps}_{\bs}  }^4 \leq C \pare{\bt - \bs}^2,
\]
and hence the sequence is tight.

Step (ii). Comparison of the limits.\\
We can now use a Gronwall argument to show that $\bt \mapsto \tilde{X}^{c,\eps}_{\bt}$ and $\bt \mapsto \tilde{\Xi}^{\eps}_{\bt}$ have the same limit (in distribution) when $\eps \to 0$. First, since the coupled process is tight, we can use the Skorokhod embedding theorem, and consider a new probability space where an $\eps$-sequence converge almost surely (in uniform norm) towards a limiting process $\bt \mapsto (\tilde{X}^{c,0}_{\bt},\tilde{\Xi}^{0}_{\bt})$. Then, we compare the two processes using~\eqref{eq:deltaTc} and~\eqref{eq:Markov}. Incorporating the terms $r_1$ (see equation~\eqref{eq:sum_X_c}) and $r_2$ in the sum, we get
\begin{align*}
  \abs{\tilde{X}^{c,\eps}_{\bt} - \tilde{\Xi}^{\eps}_{\bt}} \leq\sum_{n=0}^{\lfloor  \lambda_0 \bt/\eps^2\rfloor} \pare{ \eps^2\frac{\theta_{n+1}}{\lambda_0^2} \abs{  \pare{A_\eps(X_{T^c_n}^c)-A_\eps(\Xi^c_{n})}}  + \theta_{n+1} \bigo (\eps^3 )},
\end{align*}
so that taking expectation yields
\[
\E\abs{\tilde{X}^{c,\eps}_{\bt} - \tilde{\Xi}^{\eps}_{\bt}} \leq C \sum_{n=0}^{\lfloor  \lambda_0 \bt/\eps^2\rfloor} \eps^2 \E\abs{  \pare{A_\eps(X_{T^c_n}^c)-A_\eps(\Xi^c_{n})}} + \bigo(\eps).
\]
Then, since the velocity of both processes is of $\bigo{\eps}$, and using the fact that $A_\eps$ is Lipschitz uniformly in $\eps$, we get
\[
\abs{  \pare{A_\epsilon(X_{T^c_n}^c)-A_\epsilon(\Xi^c_{n})}} \leq C \eps^{-2} \int_{n \eps^2/\lambda_0}^{(n+1) \eps^2/\lambda_0}  \abs{\tilde{X}^{c,\eps}_{\bs} - \tilde{\Xi}^{\eps}_{\bs} } \, d \bs + \theta_{n+1}\bigo(\eps).
\]
Finally we get
\begin{align*}
 \E  \abs{\tilde{X}^{c,\eps}_{\bt} - \tilde{\Xi}^{\eps}_{\bt}} \leq C \int_0^{\bt}\E\abs{\tilde{X}^{c,\eps}_{\bs} - \tilde{\Xi}^{\eps}_{\bs} }\, d \bs + \bigo(\eps),
\end{align*}
and applying a Gronwall argument yields
\begin{align*}
 \E  \abs{\tilde{X}^{c,\eps}_{\bt} - \tilde{\Xi}^{\eps}_{\bt}} \leq \bigo(\eps).
\end{align*}
Taking the limit $\eps \to 0$ shows that $\tilde{X}^{c,0}_{\bt} = \tilde{\Xi}^{0}_{\bt}$ almost surely.

Step (iii). Computation of the limit.\\
Standard results (see \cite{StrVarSri06}) show that the sequence of processes $\bt \mapsto \tilde{\Xi}^{\eps}_{\bt}$ satisfying~\eqref{eq:Markov}
converges in distribution for the uniform norm towards the process solution to the SDE with diffusion matrix
\[
a := \lambda_0 \E\pare{ \frac{\theta_{n+1}^2\Vj_n \otimes \Vj_n}{\lambda_0^2}  } = \dfrac{2 D}{\lambda_0},
\]
and drift vector field
\[
b(x) := \lambda_0 \E\pare{ \frac{\theta_{n+1}}{\lambda_0^2} \Vj_n A^T(x) \Vj_n   } = \dfrac{D A(x)}{\lambda_0},
\]
with $D$ given by \eqref{eq:maxw}, which yields the SDE~\eqref{eq:cprocess_hydro}. By Step (ii), the same holds true for~$\bt \mapsto \tilde{X}^{c,\eps}_{\bt}$.

Step (iv). Random time change.\\
It remains to show that $\bt \to \tilde{X}^{c,\eps}_{\bt}$ and $\bt \to X^{c,\eps}_{\bt} $ converge in distribution to the same limit. Let us introduce the random time change
\begin{align}\label{eq:timechange_c}
   \alpha_{c,\eps}(\bt) & =  \eps^2 T_{\lfloor  \lambda_0 \bt/\eps^2\rfloor}^{c} + \eps^2 \pare{T_{\lfloor  \lambda_0 \bt/\eps^2\rfloor +1}^{c}  - T_{\lfloor  \lambda_0 \bt/\eps^2\rfloor}^{c}}\pare{ \lambda_0 \bt/\eps^2 - \lfloor  \lambda_0 \bt/\eps^2\rfloor } \\
& = \sum_{n=1}^{\lfloor  \lambda_0 \bt/\eps^2\rfloor} \eps^2 \Delta T_{n}^{c}  +  \eps^2\pare{\Delta T_{\lfloor  \lambda_0 \bt/\eps^2\rfloor +1}^{c}  }\pare{ \lambda_0 \bt/\eps^2 - \lfloor  \lambda_0 \bt/\eps^2\rfloor },
\end{align}
which is exactly the linear interpolation of the jump times $\pare{T^c_n}_{n\geq 0}$, and by construction is almost surely strictly increasing and continuous. Using \eqref{eq:deltaTc}, we get
\[
 \alpha_{c,\eps}(\bt) = \frac{\eps^2}{\lambda_0}\sum_{n=1}^{\lfloor  \lambda_0 \bt/\eps^2\rfloor}  \theta_{n}  + \frac{\bigo (\eps^3)}{\lambda_0}\sum_{n=1}^{\lfloor  \lambda_0 \bt/\eps^2\rfloor}  (\theta_{n} +1),
\]
so that $\bt \mapsto \alpha_{c,\eps}(\bt)$ converges in distribution on any time interval $[0,T]$ for the uniform convergence topology towards the deterministic function $\bt \mapsto \bt$ (see e.g. \cite{Kus84}). By construction,
\begin{align}\label{eq:alphatilde_c}
X^{c,\eps}_{\alpha_{c,\eps}(\bt)} &= \tilde{X}^{c,\eps}_{\bt},
\end{align}
so that we can apply the technical Lemma~\ref{l:techrandt} to conclude.
\end{proof}

\begin{rem}
	In one spatial dimension, one can describe the evolution of the distribution density $p^c(x,v,t)$ of the process with direct gradient sensing as follows \cite{ErbOthm04}. We introduce the distributions of left-moving and right-moving particles $p^c_{\pm}(x,t)=p^c(x,v=\pm 1,t)$, and write, using \eqref{e:kinetic}:
\begin{equation} \label{eq:cprocess_hydro1d}
\system{
	 & \partial_t  p^c_++\epsilon\partial_x p^c_+=-\frac{\lambda^c_\eps(x,+1)}{2}p_+^c+ \frac{\lambda^c_\eps(x,-1)}{2}p_-^c &\\
	 & \partial_t  p^c_--\epsilon\partial_x p^c_-=\frac{\lambda^c_\eps(x,+1)}{2}p_+^c -\frac{\lambda^c_\eps(x,-1)}{2}p_-^c &  
	 }.
 \end{equation}
If we assume $\lambda^c_\eps(x,v)$ to be a linear function, 
\begin{equation}\label{eq:lin-turn-1D}
\lambda^c_\eps(x,v) = \lambda_0 - \eps A(x)  v,
\end{equation}
is easy to see that equation \eqref{eq:cprocess_hydro1d} is equivalent to 
\begin{equation}
\partial^2_t n^c + \lambda_0 \partial_t n^c = \partial_x \pare{\epsilon^2 \partial_x n^c - \eps^2 A(x) n^c},
\end{equation}
which, on diffusive time-scales $\bar{t}=t \, \eps^2$, is equivalent to \eqref{eq:cdens_hydro} in the limit when $\eps\to 0$. 
Note, however, that the assumption \eqref{eq:lin-turn-1D} requires $|A(x)|\le\lambda_0/\eps$ to retain positivity of the turning rate.  For $\epsilon$ small enough compared to $A(x)$ (containing $\nabla S(x)$) this requirement is always fulfilled.
\end{rem}

\subsection{Diffusion limit of the process with internal state\label{sec:limit-internal}}

In the same way, a standard probabilistic diffusion approximation argument can be used to derive the pathwise diffusive limit of the process with internal state (\ref{eq:process_noscale}). We denote by 
\[
X^\eps_{\bar{t}} = X_{t/\eps^2},
\]
the process with internal state on diffusive time scales. We will show that this process converges towards a solution of the same advection-diffusion SDE (\ref{eq:cprocess_hydro}) as the process with direct gradient sensing \eqref{eq:cprocess} for $\eps \to 0$, provided $A(x)$ is chosen according to 
\begin{equation}\label{eq:control_field}
A_0(x) = b^T \lim_{\eps \to 0}\frac{\taueps}{\lambda_0 \taueps + \Id } \nabla S(x),
\end{equation}
assuming that the limit exists. The parameters $b$, $\taueps$, and $\lambda_0$ in equation~\eqref{eq:control_field} were introduced in (\ref{eq:tau_eps})-(\ref{eq:lin_rate}) as parameters of the process with internal state, and $\Id \in \R^{n\times n}$ is the identity matrix. 
If we introduce the bacterial density as
\begin{equation}
	n(x,t)= \int_{\Y}\int_{\V}p(x,v,y,t) \Max(\d v) \d y,
\end{equation}
this implies that $n$ satisfies (\ref{eq:cdens_hydro}) on diffusive time scales as $\epsilon \to 0$. 

The proof proceeds along the same lines as the proof of Proposition~\ref{pro:convcoupl}.  However, the derivation of an asymptotic expansion of the jump time differences becomes more involved. 
We need the following auxiliary definitions and lemmas~:
\begin{defn}\label{def:m}
  We denote by $m:\R \to \R^{n \times n}$ the function
  \begin{equation}
    \label{eq:m(t)}
    m(t) := t \taueps - \pare{\Id - {\rm e}^{-t \taueps^{-1}}} \taueps^2,
  \end{equation}
whose derivative is given by
\[
m'(t) = \taueps\pare{\Id - {\rm e}^{-t \taueps^{-1}} }.
\]
\end{defn}

\begin{lem}
For all $t \in \R$, we have
\[
\norm{ m'(t) } \leq t ,
\]
as well as
\[
\norm{ m(t) } \leq \frac{t^2}{2} .
\]
\end{lem}
\begin{proof}
We have $m(0)=m'(0)=0$ and $m''(t) = \e^{-t \taueps^{-1}}$. By Assumption \ref{a:1} on $\taueps$, we have
\[
\norm{m''(t)} \leq 1,
\]
so that
\[
\norm{m'(t)} = \norm{\int_0^t m''(s) \, ds} \leq \int_0^t \, ds=t,
\]
as well as
\[
\norm{m(t)} = \norm{\int_0^t m'(s) \, ds} \leq \int_0^t s\, ds =\dfrac{t^2}{2}.
\]
\end{proof}
\begin{lem}\label{lem:DTasymp} Let us denote ($\nabla^2$ denotes the Hessian) as
\begin{equation}\label{eq:nonlin_signal}
c_S = \norm{\nabla^{2}  S}_{\infty},
\end{equation}
as well as the error term due to non-linearities of the problem as (recall that by assumption $\delta < 1/k$)
\begin{equation}
  \label{eq:nl_err}
  \nl(\eps) = c_{\lambda} \eps^{ k \delta} + c_{F} \eps^{1+\delta} + c_S \eps^{2} < \bigo( \eps ),
\end{equation}
where the constants are defined in resp.~equations \eqref{eq:lin_rate}, \eqref{eq:tau_eps}, and \eqref{eq:nonlin_signal}. 
The difference between two jump times, 
\[
\Delta T_{n+1} := T_{n+1}  - T_{n}, 
\]
satisfies the equation
\begin{align}\label{eq:DT1}
\begin{split}
   \theta_{n+1}&=\lambda_0 \Delta T_{n+1} - b^T m'(\Delta T_{n+1}) Z_{T_n}  \\
& \quad - \eps b^T m(\Delta T_{n+1}) \nabla S( X_{T_n}) \Vj_n + \pare{\theta_{n+1}^3+ \theta_{n+1}}\bigo \pare{ \nl(\eps) }.
\end{split}                                                                                                                                                                                       
\end{align}
\end{lem}
\begin{proof} 
Throughout the proof, we will use the estimate $\abs{\Delta T_{n+1} } \leq C \theta_{n+1}$, without always mentioning explicitly.
Duhamel's integration of the ODE in~\eqref{eq:process_noscale} on $[T_n,T_{n+1}]$ yields
  \begin{align}\label{eq:Duhamel_lin_2}
    Z_t = & \e^{- \pare{t-T_n}{\taueps^{-1}}} Z_{T_n} + \eps \int_{T_n}^{t} \e^{- \pare{t-t'}{\taueps^{-1}}} \nabla S(X_{t'})  V_{t'} \d t' \nonumber  \\
& + \int_{T_n}^{t} \e^{- \pare{t-t'}{\taueps^{-1}}} \pare{\taueps^{-1}Z_{t'} - \Feps(Y_{t'},S(X_{t'})} \, d t' .
%+ \theta_{n+1} \bigo \pare{\eps^2}~
  \end{align}
Since $X_{t'}=X_{T_n}+\eps(t'-T_n)\Vj_n$, with $t' \leq T_{n+1}$, we get,
\[
\abs{ \nabla S(X_{t'}) - \nabla S(X_{T_n}) } = \theta_{n+1} c_S \bigo(\eps ),
\]
using $c_S$ as defined in \eqref{eq:nonlin_signal} and the estimate $t'-T_n\le \Delta T_{n+1}\le C\theta_{n+1}$.
Moreover, Assumption~\ref{a:1}, combined with Assumption~\ref{ass:ODE}, ensures that
\[
\abs{ \taueps^{-1}Z_{t'} - \Feps(Y_{t'},S(X_{t'}) } = \eps^{1-\delta}c_F\bigo(\abs{Z_{t'}}^{2\delta}) = c_F\bigo(\eps^{1+\delta} ).
\]
 
Inserting these two estimates in~\eqref{eq:Duhamel_lin_2}, yields  \begin{align}\label{eq:key_Duhamel}
  Z_t &= \e^{- \pare{t-T_n}{\taueps^{-1}}} Z_{T_n} + \eps \int_{T_n}^{t} \e^{- \pare{t-t'}{\taueps^{-1}}} \d t' \nabla S(X_{T_n})  \Vj_{T_n} + (\theta_{n+1}^2+ \theta_{n+1})\bigo \pare{ c_{F} \eps^{1+\delta} + c_S \eps^{2} } \nonumber \\
&= \e^{- \pare{t-T_n}{\taueps^{-1}}} Z_{T_n} + \eps \pare{Id - \e^{- \pare{t-T_n}{\taueps^{-1}}} }{\taueps} \nabla S(X_{T_n})  \Vj_{T_n} + (\theta_{n+1}^2+ \theta_{n+1})\bigo \pare{ c_{F} \eps^{1+\delta} + c_S \eps^{2} }.
\end{align}
Integrating again, we get
\begin{align*}
  \begin{split}
   \int_{T_n}^{T_{n+1}} Z_t \d t &= \pare{Id - \e^{- \pare{T_{n+1}-T_n}{\taueps^{-1}}} }{\taueps} \pare{Z_{T_n} -\eps\taueps \nabla S(X_{T_n}) \Vj_{T_n} } \\
& \quad +\eps (T_{n+1}-T_n)\taueps \nabla S(X_{T_n})  \Vj_{T_n} + (\theta_{n+1}^3+\theta_{n+1}^2) \bigo \pare{c_{F} \eps^{1+\delta} + c_S \eps^{2} }  \\
&= m'(\Delta T _{n+1})Z_{T_n} +
 \eps m(\Delta T _{n+1})\nabla S(X_{T_n}) \Vj_{T_n} + (\theta_{n+1}^3+\theta_{n+1}^2) \bigo \pare{c_{F} \eps^{1+\delta} + c_S \eps^{2} },
\end{split}
\end{align*}
with $m(t)$ and $m'(t)$ defined as in Definition~\ref{def:m}.
Finally, \eqref{eq:lin_rate} gives
\[
\theta_{n+1} = \lambda_0 \Delta T_{n+1} - b^T  \int_{T_n}^{T_{n+1}} Z_t dt +\theta_{n+1} c_\lambda \bigo \pare{\eps^{k\delta}},
\]
and the result~\eqref{eq:DT1} follows.
\end{proof}

The following estimate of jump times then follows~:
\begin{lem}\label{lem:DTestim}
The jump time variations can be written in the following form~:
\begin{equation}\label{eq:DTestim}
  \Delta T_{n+1} = \Delta T_{n+1}^0 + \eps \Delta T_{n+1}^1 + (\theta_{n+1}^6+\theta_{n+1})\bigo \pare{\nl(\eps) + \eps^2},
\end{equation}
where 
\begin{align}
\Delta T_{n+1}^0 &= \frac{\theta_{n+1}}{\lambda_0} + \frac{b^T}{\lambda_0} m'(\Delta T_{n+1}^0) Z_{T_n},  \label{eq:estimDT0} \\   
 \Delta T_{n+1}^1 &= \frac{1}{\lambda_0 - b^T \e^{- {\Delta T_{n+1}^0}{\taueps^{-1}}} Z_{T_n} } b^T m(\Delta T_{n+1}^0 ) \nabla  S(X_{T_n})  \Vj_{n}  \label{eq:estimDT1}, 
\end{align}
or, for short,
\begin{align}
\Delta T_{n+1}^0 &= \frac{\theta_{n+1}}{\lambda_0} + \theta_{n+1} \bigo(\eps^{\delta}) \label{eq:estimDT0_bis} \\
\Delta T_{n+1}^1 & = \frac{1}{\lambda_0}   b^T m\left(\frac{\theta_{n+1}}{\lambda_0}  \right) \nabla S(X_{T_n}) \Vj_{n} +  \theta_{n+1}^2\bigo (\eps^{\delta}),\label{eq:estimDT1_bis}
\end{align}
where the term of order $\eps^\delta$ in~\eqref{eq:estimDT0_bis} is independent of $\Vj_n$.
\end{lem}
\begin{proof}
 Let us first introduce the functions
\[
\begin{cases}
 \Psi_0(t) = \lambda_0 t - b^T m'(t) Z_{T_n} \\
 \Psi_1(t) =  - b^T m(t)  \nabla S( X_{T_n}) \Vj_n,
\end{cases}
\]
Using these functions, we can rewrite the result \eqref{eq:DT1} of Lemma~\ref{lem:DTasymp}, as
\begin{equation}
  \label{eq:DT1bis}
  \theta_{n+1} =  \Psi_0(\Delta T_{n+1}) + \eps \Psi_1( \Delta T_{n+1} ) + \pare{\theta_{n+1}^3+\theta_{n+1}}\bigo \pare{ \nl(\eps) }.
\end{equation}
Also, the definitions~\eqref{eq:estimDT0}-\eqref{eq:estimDT1} of $(\Delta T_{n+1}^0,\Delta T_{n+1}^1)$ can be written as
\begin{equation}\label{eq:psi_DT}
  \begin{cases}
    \Psi_0(\Delta T_{n+1}^0) = \theta_{n+1} \\
     \Delta T_{n+1}^1 = -\dfrac{\Psi_1(\Delta T_{n+1}^0)}{\Psi_0'(\Delta T_{n+1}^0)}.
  \end{cases}
\end{equation}
Taking the first line, add $\epsilon$ times the second line, and rearrange, one has 
\begin{equation}\label{eq:why_this}
\theta_{n+1} = \Psi_0(\Delta T_{n+1}^0) + \eps \Psi_0'(\Delta T_{n+1}^0) \Delta T_{n+1}^1 + \eps \Psi_1(\Delta T_{n+1}^0).
\end{equation}
We then consider the expansion
\[
\Delta T_{n+1} = \Delta T_{n+1}^0 + \eps \Delta T_{n+1}^1 + R_{n+1}.
\]
The proof of the Lemma then reduces to showing that $R_{n+1}$ is small. Anticipating the result, the idea is to Taylor expand $\Psi_0(\Delta T_{n+1})$ and $\Psi_1(\Delta T_{n+1})$, use these Taylor expansions in~\eqref{eq:psi_DT}, and compare to~\eqref{eq:why_this}. We start by using a double Taylor expansion at order~$1$ for $\Psi_0(\Delta T_{n+1})$~:
\begin{align*}
  \Psi_0(\Delta T_{n+1}) =& \Psi_0(\Delta T_{n+1}^0+R_{n+1})+ \eps \Delta T_{n+1}^1 \Psi_0'(\Delta T_{n+1}^0+R_{n+1}) \\
& + \frac{(\eps \Delta T_{n+1}^1)^{2}}{2}  \Psi_0''(\Delta T_{n+1}^0 +   R_{n+1} + \gamma_3\eps \Delta T_{n+1}^1 ), \\
=& \Psi_0(\Delta T_{n+1}^0)+ R_{n+1} \Psi_0'(\Delta T_{n+1}^0 + \gamma_1 R_{n+1}) + \eps \Delta T_{n+1}^1\Psi_0'(\Delta T_{n+1}^0) 
 \\
&+ \eps \Delta T_{n+1}^1 R_{n+1}  \Psi_0''(\Delta T_{n+1}^0 + \gamma_2  R_{n+1}  ) + \frac{(\eps \Delta T_{n+1}^1)^{2}}{2}  \Psi_0''(\Delta T_{n+1}^0 +   R_{n+1} + \gamma_3 \eps \Delta T_{n+1}^1 ) ,
\end{align*}
 and a simple Taylor expansion at order~$0$ for $\Psi_1(\Delta T_{n+1})$:
\[
\Psi_1(\Delta T_{n+1}) = \Psi_1(\Delta T_{n+1}^0) + (\eps \Delta T_{n+1}^1 + R_{n+1})\Psi_1'(\Delta T_{n+1}^0+ \gamma_{4} (\eps \Delta T_{n+1}^1 + R_{n+1})) ,
\]
where $\gamma_i\in [0,1], i = 1 \dots 4$ are used to apply the intermediate value theorem to the rest of the Taylor expansions. 
Inserting the Taylor expansions in the identity~\eqref{eq:DT1bis}, and comparing to~\eqref{eq:why_this} simplifies to
\begin{align}
  \label{eq:DT1ter}
  0 = & \Psi_0'(\Delta T_{n+1}^0 + \gamma_1 R_{n+1} )R_{n+1}  \nonumber \\
& \quad + \eps \Delta T_{n+1}^1 R_{n+1}   \Psi_0''(\Delta T_{n+1}^0 + \gamma_2  R_{n+1} ) + \frac{\pare{\eps \Delta T_{n+1}^1}^2}{2} \Psi_0''(\Delta T_{n+1}^0 +   R_{n+1} + \gamma_3 \eps \Delta T_{n+1}^1 )   \nonumber \\
& \quad +  (\eps^{2} \Delta T_{n+1}^1 + \eps R_{n+1}) \Psi_1'(\Delta T_{n+1}^0+ \gamma_{4} (\eps \Delta T_{n+1}^1 + R_{n+1}) )+\pare{\theta_{n+1}^3+\theta_{n+1}}\bigo \pare{\nl(\eps)},
\end{align}

We now need some estimates on the terms involved in~\eqref{eq:DT1ter} to control~$R_{n+1}$. First of all, we need some crude bounds on jump times. Since rates are bounded above and below by $\lambda_{\rm min}$ and $\lambda_{\rm max}$, we have
\[
\abs{\Delta T_{n+1}} \leq C \theta_{n+1}.
\] 
Then, $\Psi_0'(t) = \lambda_0 - \lambda_0b^T {\rm e}^{-t \taueps^{-1}}Z_{T_n}$, so that, using Assumption~\ref{ass:ODE}, there is a constant $C$ such that
\[
 \frac{1}{2} \lambda_0 \leq \sup_{t \in [0, +\infty[}\Psi_0'(t) \leq \frac{3}{2} \lambda_0,
\]
for $\eps \leq 1/C$. As a consequence, $\Delta T_{n+1}^0$ is well-defined by the implicit definition~\eqref{eq:estimDT0} and one has obviously
\[
\abs{\Delta T_{n+1}^0} \leq C \theta_{n+1}.
\]
In the same way
\[
\abs{\Delta T_{n+1}^1} \leq C m(\Delta T_{n+1}^0) \leq C \theta_{n+1}^2,
\]
and thus by definition:
\[
\abs{R_{n+1}} \leq C ( \theta_{n+1}+\eps\theta_{n+1}^2).
\]
Then we need some bounds on the derivatives of $\Psi_0$ and $\Psi_1$. Using Lemma~\ref{eq:techbound}, we get
\[
\abs{\Psi_0''(t)} = \lambda_0 \abs{b^T \taueps^{-1} {\rm e}^{-t \taueps^{-1}}Z_{T_n} } \leq C.
\]
Eventually, for any $t \geq 0$,
\[
\abs{\Psi_1'(t)} \leq C m'(t) \leq C t^2.
\]
Inserting all the previous estimates in~\eqref{eq:DT1ter} to control $R_{n+1}$, we get
\begin{align*}
  %\label{eq:DT1quatr}
    \abs{R_{n+1}} & \leq C \eps ( \theta_{n+1}^4 + \theta_{n+1}^2) + \pare{\theta_{n+1}^3+ \theta_{n+1}} \bigo \pare{\nl(\eps)}.
\end{align*}
Inserting a second time in~\eqref{eq:DT1ter} to control $R_{n+1}$, it yields:
\[
\abs{R_{n+1}} \leq C \eps^2 ( \theta_{n+1}^6 + \theta_{n+1}^2) + \pare{\theta_{n+1}^5+ \theta_{n+1}} \bigo \pare{\nl(\eps)},
\] 
and the result follows.
\end{proof}

\begin{rem}It is crucial to note that the term of order $\eps^\delta$ in~\eqref{eq:estimDT0}-\eqref{eq:estimDT0_bis} is independent of $\Vj_n$. This will ensure that the latter contributes to martingale terms in sums of the form~$\sum_n \Delta T _n \Vj_n$.
\end{rem}

Finally, the drift of the limiting diffusion (see Proposition below) will be computed using the following lemma~:
\begin{lem}\label{lem:av_m}
  Let $\theta$ be an exponential random variable of mean $1$. Then:
  \begin{equation}
    \label{eq:av_m}
    \E \pare{ m\left( \frac{\theta}{\lambda_0} \right)  } = \frac{1}{\lambda_0} \frac{\taueps}{\Id + \lambda_0 \taueps}.
  \end{equation}
\end{lem}
\begin{proof}
  It is a straightforward computation:
  \begin{align*}
    \E \pare{ m\left( \frac{\theta}{\lambda_0}\right)   }=& \int_0^{+\infty}m(t/\lambda)e^{-t}dt\\ =& \frac{\taueps}{\lambda_0} \int_0^{+\infty}t {\rm e}^{-t}\, dt - \taueps^2\int_0^{+\infty}{\rm e}^{-t}\, dt + \taueps^2 \int_0^{+\infty}{\rm e}^{-(\Id+ \frac{\taueps^{-1}}{\lambda_0}   )t}\, dt \\
 =&\frac{\taueps}{\lambda_0}  - \taueps^2 +\taueps^2 \pare{\Id+ \taueps^{-1}/\lambda_0}^{-1} \\
=&\frac{\taueps}{\Id + \lambda_0 \taueps} \pare{ \frac{1}{\lambda_0}\pare{\Id + \lambda_0 \taueps}\pare{\Id - \lambda_0 \taueps}     + \lambda_0\taueps^2       } \\
=& \frac{1}{\lambda_0} \frac{\taueps}{\Id + \lambda_0 \taueps}.
  \end{align*}
\end{proof}
Using these lemmas, we can prove the following result on the diffusive limit of the process with internal state~:
\begin{pro}\label{pro:convproc}
Consider the process with internal state define in Section~\ref{sec:model-internal}, and the assumptions of Section~\ref{sec:ass_ass}.
Assume $\lim_{\eps \to 0}\dfrac{\taueps}{\lambda_0 \taueps + \Id }$ exists, and that $ \delta > 1/k$. Then, the process $\bt \mapsto X_{\bar{t}}^{\eps} = X_{t/\eps^2} $ converges in distribution (for uniform convergence topology) towards $\bt \mapsto X_{\bar{t}}^{0}$ solution to the SDE~\eqref{eq:cprocess_hydro} with drift coefficient~\eqref{eq:control_field}. As a consequence, the density $n^\eps_{\bar{t}}$ of the process converges weakly towards $n^0_{\bar{t}}= n^{c,0}_{\bar{t}}$ solution of (\ref{eq:cdens_hydro}) with drift coefficient~\eqref{eq:control_field}.
\end{pro}
\begin{proof}
Since the proof is very similar to the proof of Proposition~\ref{pro:convcoupl}, we will skip similar arguments and refer to the latter whenever possible.
As in the proof of Proposition~\ref{pro:convcoupl}, we perform a random time change on the process of positions of the bacterium, so that a jump (tumble phase) become deterministic, and will occur on regular small time interval of length $\eps^2/\lambda_0$. The latter is thus defined as:
\begin{equation}
  \label{eq:sum_X}
  \tilde{X}^{\eps}_{\bt} 
= X_ 0 + \sum_{n=0}^{\lfloor  \lambda_0 \bt/\eps^2\rfloor - 1} \eps \Delta T_{n+1} \Vj_{n} + \underbrace{\pare{ \lambda_0 \bt/\eps^2 -\lfloor  \lambda_0 \bt/\eps^2\rfloor}\eps \Delta T_{\lfloor  \lambda_0 \bt/\eps^2\rfloor+1} \Vj_{\lfloor  \lambda_0 \bt/\eps^2\rfloor}}_{r_1}.
\end{equation}
To compute the diffusion limit $\tilde{X}^{\eps}_{\bt}$, we will compare it to a classical drifted random walk defined by the Markov chain
\begin{equation}
  \label{eq:Markov_int}
  \begin{cases}
    \Xi_0= X_ 0 \\
\Xi_{n+1} = \Xi_{n} + \eps \frac{\theta_{n+1} \Vj_n}{\lambda_0} + \eps^2 \frac{\Vj_n}{\lambda_0} b^T m(\frac{\theta_{n+1}}{\lambda_0}  ) \nabla S(\Xi_{n}) \Vj_{n}.
  \end{cases}
\end{equation}
To the latter, we can associate a continuous process defined on diffusive time scales by interpolation
\begin{equation}
  \label{eq:sum_Xi}
  \tilde{\Xi}^{\eps}_{\bt} 
= X_ 0 + \pare{ \sum_{n=0}^{\lfloor  \lambda_0 \bt/\eps^2\rfloor - 1} \Xi_{n+1}-\Xi_n} + \underbrace{\pare{ \lambda_0 \bt/\eps^2 -\lfloor  \lambda_0 \bt/\eps^2\rfloor}\pare{ \Xi_{\lfloor  \lambda_0 \bt/\eps^2\rfloor + 1}-\Xi_{\lfloor  \lambda_0 \bt/\eps^2\rfloor } }}_{r_2}.
\end{equation}
Now the computation of the diffusive limit of the sequence of processes $ \bt \mapsto  \tilde{X}^{\eps}_{\bt}$ relies again on four steps. \\

Step (i). Tightness of the coupled sequence $\bt \mapsto (\tilde{X}^{\eps}_{\bt},\tilde{\Xi}^{\eps}_{\bt})$. \\
The proof of tightness is strictly similar to the proof in Proposition~\ref{pro:convcoupl}, so we leave it to the reader.

Step (ii). Comparison of the limits.\\
We can now use a Gronwall argument to show that $\bt \mapsto \tilde{X}^{\eps}_{\bt}$ and $\bt \mapsto \tilde{\Xi}^{\eps}_{\bt}$ have the same limit (in distribution) when $\eps \to 0$. First, since the coupled process is tight, we can consider a new probability space where a $\eps$-sequence converges almost surely (in uniform norm) towards the limiting process $\bt \mapsto (\tilde{X}^{0}_{\bt},\tilde{\Xi}^{0}_{\bt})$. Then, we compare the two processes using~\eqref{eq:DTestim}-\eqref{eq:sum_X} and~\eqref{eq:Markov_int}-\eqref{eq:sum_Xi}. For this purpose, we incorporate the rest terms $r_1$ and $r_2$ in the sum in~\eqref{eq:sum_X} and~\eqref{eq:sum_Xi}, and we get the estimate 
\begin{align*}
&  \abs{\tilde{X}^{\eps}_{\bt} - \tilde{\Xi}^{\eps}_{\bt} } \leq \abs{ \sum_{n=0}^{\lfloor  \lambda_0 \bt/\eps^2\rfloor} \eps \Vj_n\pare{ \Delta T ^0_{n+1} - \frac{\theta_{n+1}}{\lambda_0}} }  \\
&\quad + \sum_{n=0}^{\lfloor  \lambda_0 \bt/\eps^2\rfloor} \eps^2\Vj_n\pare{ \Delta T ^1_{n+1} - \frac{b^T}{\lambda_0} m(\frac{\theta_{n+1}}{\lambda_0}) \nabla S(\Xi_{n})\Vj_n } + \sum_{n=0}^{\lfloor  \lambda_0 \bt/\eps^2\rfloor}  (\theta_{n+1}^6+\theta_{n+1}) \bigo (\eps^{3}+\eps \nl(\eps) ).
\end{align*}
Denoting
\[
M_n = \Vj_n b^T m'(\Delta T_{n+1}^0) Z_{T_n} ,
\]
and using the estimates in Lemma~\ref{lem:DTestim}, we get
\begin{align*}
&  \abs{\tilde{X}^{\eps}_{\bt} - \tilde{\Xi}^{\eps}_{\bt}} \leq \abs{\sum_{n=0}^{\lfloor  \lambda_0 \bt/\eps^2\rfloor} \eps M_n} \\
&\quad + \sum_{n=0}^{\lfloor  \lambda_0 \bt/\eps^2\rfloor} \eps^2\frac{\theta_{n+1}}{\lambda_0} \abs{b^T m(\frac{\theta_{n+1}}{\lambda_0})}\abs{  \nabla S(X_{T^c_n})-\nabla S(\Xi_{n}) }  + \sum_{n=0}^{\lfloor  \lambda_0 \bt/\eps^2\rfloor}  (\theta_{n+1}^6+\theta_{n+1}) \bigo (\eps^{3} + \eps^{2+\delta}+\eps \nl(\eps) ) .
\end{align*}
Taking expectation and using Jensen's inequality, as well as the fact that $S$ is Lipschitz, then gives 
\begin{align*}
&  \E\abs{\tilde{X}^{\eps}_{\bt} - \tilde{\Xi}^{\eps}_{\bt}} \leq \pare{ \E \abs{\sum_{n=0}^{\lfloor  \lambda_0 \bt/\eps^2\rfloor} \eps M_n}^2}^{1/2} \\
&\quad + C\eps^2 \sum_{n=0}^{\lfloor  \lambda_0 \bt/\eps^2\rfloor} \E\abs{  X_{T^c_n}- \Xi_{n} }  + \bigo (\eps+\eps^{\delta}+ \eps^{-1} \nl(\eps) ) .
\end{align*}
Using $\abs{M_n}\leq C \theta_{n+1} \eps^{\delta}$ (which follows from Assumption~\ref{ass:ODE}) and the independence of the random variables $(\Vj_n)_{n \geq 0}$, we get
\[
\E \abs{\sum_{n=0}^{\lfloor  \lambda_0 \bt/\eps^2\rfloor} \eps M_n} ^2 \leq C \lfloor  \lambda_0 \bt/\eps^2\rfloor \times \eps^2 \times \eps^{2\delta} = \bigo( \eps^{2\delta} ).
\]
Controlling the derivative of both processes on the time interval $[n \eps^2/\lambda_0,(n+1) \eps^2/\lambda_0]$ yields
\[
\abs{  \tilde{X}_{T^c_n}-\Xi_{n} } \leq \lambda_0\eps^{-2} \int_{n \eps^2/\lambda_0}^{(n+1) \eps^2/\lambda_0}  \abs{\tilde{X}^{\eps}_{\bs} - \tilde{\Xi}^{\eps}_{\bs} } \, d \bs + \theta_{n+1}\bigo(\eps).
\]
In the end, we get
\begin{equation*}
 \E  \abs{\tilde{X}^{\eps}_{\bt} - \tilde{\Xi}^{\eps}_{\bt}} \leq C \int_0^{\bt}\E \abs{\tilde{X}^{\eps}_{\bs} - \tilde{\Xi}^{\eps}_{\bs}}\, d \bs  + \bigo (\eps+\eps^{\delta}+ \eps^{-1} \nl(\eps) )
\end{equation*}
and applying a Gronwall argument yields (expliciting $\nl(\eps)$):
\begin{equation}\label{eq:final_bound}
 \E  \abs{\tilde{X}^{\eps}_{\bt} - \tilde{\Xi}^{\eps}_{\bt}} = \bigo( (1+c_{S})\eps + (1+ c_{F})\eps^{\delta} + c_{\lambda} \eps^{ k \delta - 1} ).
\end{equation}
Taking the limit $\eps \to 0$ shows that $\tilde{X}^{0}_{\bt} = \tilde{\Xi}^{0}_{\bt}$ almost surely.

Step (iii). Computation of the limit.\\
Standard results (see \cite{StrVarSri06}) show that the sequence of processes $\bt \mapsto \tilde{\Xi}^{\eps}_{\bt}$ satisfying~\eqref{eq:Markov_int} converges in distribution for the uniform norm towards the process that is solution of the SDE with diffusion matrix~:
\[
a = \lambda_0 \E\pare{ \frac{\theta_{n+1}^2\Vj_n \otimes \Vj_n}{\lambda_0^2}  } = 2 D/\lambda_0,
\]
and drift vector field, using~\eqref{eq:av_m}, see Lemma~\ref{lem:av_m}~:
\begin{align*}
 \dps b(x) =&  \lim_{\eps \to 0}\lambda_0 \E\pare{ \frac{\Vj_n}{\lambda_0} b^T m(\frac{\theta_{n+1}}{\lambda_0}  ) \nabla S(x) \Vj_{n} } \\
  =&  \lim_{\eps \to 0} D \pare{b^T \frac{\taueps}{\Id + \lambda_0 \taueps} \nabla S(x) } / \lambda_0 , \\
 & \xrightarrow{\eps \to 0} D A(x)  / \lambda_0,
\end{align*}
with $A(x)$ defined in equation~\ref{eq:control_field}, yielding the SDE~\eqref{eq:cprocess_hydro}. By Step (ii), the same holds true for~$\bt \mapsto \tilde{X}^{\eps}_{\bt}$.

Step (iv). Random time change.\\
It remains to show that $\bt \to \tilde{X}^{c,\eps}_{\bt}$ and $\bt \to X^{c,\eps}_{\bt} $ converge to the same limit. We introduce the random time change
\begin{align}\label{eq:timechange}
   \alpha_{\eps}(\bt) & =  \eps^2 T_{\lfloor  \lambda_0 \bt/\eps^2\rfloor} + \eps^2 \pare{T_{\lfloor  \lambda_0 \bt/\eps^2\rfloor +1}  - T_{\lfloor  \lambda_0 \bt/\eps^2\rfloor}}\pare{ \lambda_0 \bt/\eps^2 - \lfloor  \lambda_0 \bt/\eps^2\rfloor } \\
& = \sum_{n=1}^{\lfloor  \lambda_0 \bt/\eps^2\rfloor} \eps^2 \Delta T_{n}  +  \eps^2\pare{\Delta T_{\lfloor  \lambda_0 \bt/\eps^2\rfloor +1}  }\pare{ \lambda_0 \bt/\eps^2 - \lfloor  \lambda_0 \bt/\eps^2\rfloor },
\end{align}
which is the linear interpolation of the jump times $\pare{T_n}_{n\geq 0}$, and by construction is almost surely strictly increasing and continuous. Lemma~\ref{lem:DTestim} implies that $\Delta T_{n}=\frac{\theta_{n+1}}{\lambda_0}+ (\theta_{n+1}^6+\theta_{n+1}) \bigo \pare{\eps + \eps^{1+\delta} + \nl(\eps)} )$.
Therefore, we get
\[
 \alpha_{\eps}(\bt) = \frac{\eps^2}{\lambda_0}\sum_{n=1}^{\lfloor  \lambda_0 \bt/\eps^2\rfloor} \theta_{n}   + \frac{\bigo (\eps^{2 + \delta} + \eps^{3}+ \eps^2\nl(\eps))}{\lambda_0}\sum_{n=1}^{\lfloor  \lambda_0 \bt/\eps^2\rfloor}  (\theta_{n}^6 +\theta_{n+1}),
\]
so that $\bt \mapsto \alpha_{c,\eps}(\bt)$ converges in distribution on any time interval $[0,T]$ for the uniform convergence topology towards the deterministic function $\bt \mapsto \bt$ (see, e.g., \cite{Kus84}). By construction,
\begin{align}\label{eq:alphatilde}
X^{\eps}_{\alpha_{\eps}(\bt)} &= \tilde{X}^{\eps}_{\bt},
\end{align}
so that we can apply the technical Lemma~\ref{l:techrandt} to conclude.
\end{proof}
 
\begin{rem}
In one spatial dimension and with internal dynamics given by \eqref{e:scalar-y},
Erban and Othmer argued on formal arguments (a moment expansion) that the 
evolution of the density $n(x,t)$ on hyperbolic time-scales satisfies
\begin{equation}\label{eq:process_hydro}
\partial^2_t n + \lambda_0 \partial_t n = \partial_x \pare{\epsilon^2 \partial_x n - \frac{ b \epsilon^2 \taueps}{1 +  \lambda_0 \taueps}S'(x)n
},
\end{equation}
in the limit when $\epsilon \to 0$ \cite{ErbOthm04}.  
It is easy to see that equation \eqref{eq:process_hydro} is equivalent to \eqref{eq:cprocess_hydro} in this limit when choosing $A(x)$ in \eqref{eq:cprocess_hydro} according to \eqref{eq:control_field}.
\end{rem}

\begin{rem}
	Another relation between the process with internal dynamics and the process with direct gradient sensing is obtained when, for fixed $\epsilon$,  the bacteria learn the value of the chemoattractants concentrations infinitely fast. To investigate this limit, we may construct a sequence $(\taueps^k)_{k=1}^{\infty}$, such that $\lim_{k\to\infty}\|\taueps^k\|^{-1} = \infty$, and show that, for $k\to\infty$, the process with internal state converges to a process with direct gradient sensing, with jump rate given by $\lambda^c_\epsilon(x,v)=\lambda_\epsilon(\nabla S(x)v)$, i.e.\ the turning rate for the process with internal dynamics, evaluated for the deviation $z=\nabla S(x)v$. 
\end{rem}

\section{Discretization of velocity-jump processes\label{sec:discr}}

To observe the asymptotic behaviour in numerical simulations, the time discretization should approach the same diffusion limit as the time-continuous problem. This requires, in particular, that the discretized jump times are computed sufficiently accurately.  In this section, we propose a time discretization of \eqref{eq:process_noscale} and give a consistency result. The process with direct gradient sensing \eqref{eq:cprocess} can be discretized similarly.  For ease of exposition, we consider the scalar equation \eqref{e:scalar-y} for the internal state; generalization to nonlinear systems of equations is suggested in Remark~\ref{rem:nonlin_num}.
  
\subsection{Linear turning rate}

First, assume a linearized jump rate 
\begin{equation}\label{eq:jump-lin}
\lambda(z) = \lambda_0  - b^T  z.
\end{equation}
When the chemoattractant profile is linear, $S(x)=S_0+(\nabla S_0)\cdot x$, the ODE \eqref{e:scalar-y} can be explicitly solved for $Z_t$. For $t\in [T_n,T_{n+1})$, {\it i.e.} between two jumps, we then have
\begin{equation} \label{e:discr-lin}
\system{
&  X_t = X_{T_n}+\eps V_t\;(t-T_n)  &\\
&  Z_t = e^{-(t-T_n) \taueps^{-1}}  Z_{T_n}+ \eps \taueps \pare{\Id-e^{-(t-T_n)\taueps^{-1}}} (\nabla S_0) \, \Vj_n.\\
%& V_t = \Vj_n.& 
}
\end{equation}
Similarly to the computation in the proof of Lemma~\ref{lem:DTasymp}, we get, using
$
\Delta T_{n+1} := T_{n+1}  - T_{n},
$
\begin{multline}\label{e:jump-lin}
	\int_{T_n}^{T_{n+1}}\lambda(Z_t)\d t= 
	\lambda_0 \pare{\Delta T_{n+1}} - b^T \pare{Id - \e^{- {\Delta T_{n+1}}{\taueps^{-1}}} }{\taueps} Z_{T_n}  \\
- \eps b^T \pare{\Delta T_{n+1} \taueps -(\Id-\e^{-\Delta T_{n+1} \taueps^{-1} } )\taueps^2   }  (\nabla S_0) \Vj_n .
\end{multline}
The jump time $T_{n+1}$ can be calculated exactly from \eqref{e:jump-lin} by solving  
\begin{equation}
	\int_{T_n}^{T_{n+1}}\lambda(Z_t)\d t=\theta_{n+1},
\end{equation}
using a Newton iteration. Hence, no time discretization error is made.

When $S(x)$ is a nonlinear function of $x$, exact integration of $Y_t$ is no longer possible.  We therefore define a numerical solution  $(\disc{X}_t,\disc{V}_t,\disc{Y}_t)$ as follows.
Between jumps, we discretize the simulation in steps of size $\delta t$ and denote by $(\disc{X}_{n,k},\disc{Z}_{n,k})$ the solution at $t_{n,k}=\disc{T}_n+k\delta t$.  The numerical solution for $t\in[t_{n,k},t_{n,k+1}]$ is given by 
\begin{equation} \label{e:discr-nonlin}
  \begin{cases}
    \disc{X}_t = \disc{X}_{n,k}+\eps \Vj_{n}\;(t-t_{n,k})  \\
    \disc{Z}_t = \exp(-(t-t_{n,k}) \taueps^{-1}) \disc{Z}_{n,k}  + \eps \taueps \pare{\Id-\exp(-(t-t_{n,k})\taueps^{-1})} \nabla S(\disc{X}_{n,k}) \, \Vj_n.
  \end{cases}
\end{equation}
We denote by $K\geq 0$ the integer such that the simulated jump time $\disc{T}_{n+1}\in [t_{n,K},t_{n,K+1}]$. To find $\disc{T}_{n+1}$, we first approximate the integral $ \int_{\disc{T}_n}^{\disc{T}_{n+1}}\lambda(\disc{Z}_t)\d t$ using 
\begin{align}
	\int_{\disc{T}_n}^{\disc{T}_{n+1}}\lambda(\disc{Z}_t)\d t &=\sum_{k=0}^{K-1}\int_{t_{n,k}}^{t_{n,k+1}}\lambda(\disc{Z}_t)\d t 
	+ \int_{t_{n,K}}^{\disc{T}_{n+1}}\lambda(\disc{Z}_t)\d t ,
\end{align}
and then compute:
\begin{multline}
	\int_{t_{n,k}}^{t_{n,k+1}}\lambda(\disc{Z}_t)\d t= 
	\lambda_0 \delta t - b^T \pare{Id - \e^{- {\delta t}{\taueps^{-1}}} }{\taueps} \disc{Z}_{T_n}  \\
 - \eps b^T \pare{\delta t \taueps -(\Id-\e^{-\delta t \taueps^{-1} } )\taueps^2   }  \nabla S(\disc{X}_{n,k}) \Vj_n .
\end{multline}
The jump time $\disc{T}_{n+1}$ can then be computed as the solution of 
\begin{equation}\label{e:lin-Newton}
\int_{t_{n,K}}^{\disc{T}_{n+1}}\lambda(\disc{Z}_t)\d t=\theta_{n+1}- \sum_{k=0}^{K-1}\int_{t_{n,k}}^{t_{n,k+1}}\lambda(\disc{Z}_t)\d t,
\end{equation}
again using a Newton procedure. The use of \eqref{e:discr-nonlin} as a discretization scheme has a crucial consequence: the analysis performed in Section~\ref{sec:adv-diff} for the continuous time case still holds \emph{in an exact fashion} in the time discretized case. This is mainly due to the fact that the key estimate \eqref{eq:key_Duhamel} still holds. It is left to the reader to check that all the proofs performed in Section~\ref{sec:adv-diff} are not affected by this discretization. 
\begin{rem}\label{rem:nonlin_num}
  The discretization scheme \eqref{e:discr-nonlin} can be generalized to a non-linear evolution equation for the internal state. It is however necessary to check that a time discrete version of Assumption~\ref{ass:ODE} still holds in this case, and also that \eqref{eq:key_Duhamel} remains true.
\end{rem}
We also give the following consistency property, which is of independent interest~:
\begin{lem}\label{lem:jump-lin}
Assume $\disc{T}_n=T_n$, $X_{T_n}=\disc{X}_{T_n}$, $\Vj_n=\disc{V}_{T_n}$ and $\disc{Y}_{T_n}=Y_{T_n}$, then, using the time discretization \eqref{e:discr-nonlin}, we have 
\begin{align}
	\left|T_{n+1}-\disc{T}_{n+1}\right| \leq C \eps^2 \delta t \theta_{n+1}^2,
\end{align}
and consequently,
\begin{align}
	\left|X_{T_{n+1}}-\disc{X}_{\disc{T}_{n+1}}\right|\leq C \eps^3 \delta t \theta_{n+1}^2.\\
\end{align}
\end{lem}
\begin{proof}
Using the assumption $T_n=\disc{T}_n$ and $X_{T_n}=\disc{X}_{T_n}$, we get straightforwardly
\[
X_t = \disc{X}_t, \quad \forall t \in [T_n,T_{n+1} \wedge \disc{T}_{n+1}],
\]
so that
\[
\abs{\nabla S(X_{t}) -\nabla S\left(\disc{X}_{n,\left\lfloor(t_{n,0}-t')/\delta t \right\rfloor}\right)  } \leq C \delta t \eps .
\]
Consider Duhamel's formula \eqref{eq:Duhamel_lin_2}, and write a similar expression for the time discretized case for $t\in [T_n,\disc{T}_{n+1}]$~:
\begin{equation}\label{eq:Duhamel_num}
  \dps  \disc{Z}_t  = \e^{- (t-T_n){\taueps^{-1}}} \disc{Z}_{T_n} + \eps \int_{T_n}^{t} \e^{- \pare{t-t'}{\taueps^{-1}}} \nabla S\left(\disc{X}_{n,\left\lfloor(t_{n,0}-t')/\delta t \right\rfloor}\right)  \Vj_{n} \d t'.
\end{equation}
Using the assumption $\disc{Z}_{T_n} = Z_{T_n}$, the difference between the two Duhamel's integrals then yields
\begin{equation}\label{eq:dZ}
  \abs{Z_t - \disc{Z}_t} \leq C \eps^2 \delta t \theta_{n+1},
\end{equation}
for $t \in [T_n,T_{n+1} \wedge \disc{T}_{n+1}]$. 

We now proceed to find a bound on $\abs{ T_{n+1}-\disc{T}_{n+1}}$.  We start by writing
\begin{displaymath}
\int_{T_{n}}^{T_{n+1}}\lambda(Z_t)\d t = 	\int_{\disc{T}_{n}}^{\disc{T}_{n+1}}\lambda(\disc{Z}_t)\d t = \theta_{n+1};
 \end{displaymath}
and using the assumption $T_n=\disc{T}_n$, we obtain
\begin{align*}
 	-\int_{T_{n}}^{T_{n+1}\wedge \disc{T}_{n+1}}b^T&\pare{Z_t - \disc{Z}_t } \d t = 	\int_{T_{n+1}\wedge \disc{T}_{n+1}}^{\disc{T}_{n+1}}\lambda(\disc{Z}_t)\d t -\int_{T_{n+1}\wedge \disc{T}_{n+1}}^{{T}_{n+1}}\lambda({Z}_t)\d t  \\
&= \pare{\disc{T}_{n+1} - T_{n+1}\wedge \disc{T}_{n+1}} \lambda(\disc{Z}_{t^*_{1}})+ \pare{ T_{n+1} - T_{n+1}\wedge \disc{T}_{n+1}} \lambda(\disc{Z}_{t^*_{2}}) .
\end{align*}
for some values of $t^*_1\in [T_{n+1}\wedge \disc{T}_{n+1} , \disc{T}_{n+1}]$ and  $t^*_2\in [T_{n+1}\wedge \disc{T}_{n+1} , \disc{T}_{n}]$. Since $\lambda_{\rm min } \leq \lambda(z) \leq \lambda_{\rm max} $, we get
\begin{align*}
	\left| \disc{T}_{n+1}-T_{n+1} \right| & \leq C\int_{T_{n}}^{T_{n+1}\wedge \disc{T}_{n+1} }\abs{Z_t - \disc{Z}_t} \d t \\
& \leq C \theta_{n+1}^2 \eps^2  \delta t .
\end{align*}
Finally, note that 
\[
\abs{X_t-\disc{X}_t } \leq \eps (t-T_n)
\]
for $t \geq T_n$ to get the estimate on $\abs{X_{T_{n+1}}-\disc{X}_{\disc{T}_{n+1} } } $. This concludes the proof.
\end{proof}
 Note that the dependence in $\eps^3\delta t$ shows, at least at the consistency level, that the error due to discretization remains of a higher order in $\eps$ than the difference between the process with internal state and its advection-diffusion limit. 

\subsection{Nonlinear turning rate}

We now consider the nonlinear turning rate as defined in \eqref{eq:rate}. We again discretize in time to obtain the time-discrete solution \eqref{e:discr-nonlin}.  The jump time $\disc{T}_{n+1}$ is now computed by linearizing \eqref{eq:rate} in each time step, 
\begin{equation}\label{eq:nonlinrate_disc}
\disc{\lambda}\pare{\disc{Z}_t,\disc{Z}_{n,k}}=\lambda(\disc{Z}_{n,k})+ \frac{\d \lambda(\disc{Z}_{n,k})}{\d z}\pare{\disc{Z}_t-\disc{Z}_{n,k}}
\end{equation}
and approximating the integral
\begin{align}
	\int_{\disc{T}_n}^{\disc{T}_{n+1}}\disc{\lambda}(\disc{Z}_t,\disc{Z}_{n,\lfloor (t- \disc{T}_n)/\delta t \rfloor})\d t &=\sum_{k=0}^{K-1}\int_{t_{n,k}}^{t_{n,k+1}}\disc{\lambda}(\disc{Z}_t,\disc{Z}_{n,k})\d t 
	+ \int_{t_{n,K}}^{\disc{T}_{n+1}}\disc{\lambda}(\disc{Z}_t,\disc{Z}_{n,K})\d t  \nonumber \\
&= \theta_{n+1},
\end{align}
The jump time $\disc{T}_{n+1}$ can then again be computed using a Newton procedure as in \eqref{e:lin-Newton}.  Here again, 
the analysis performed in Section~\ref{sec:adv-diff} for the continuous time case, still holds \emph{in an exact fashion} in the time discretized case.
Also, a consistency property similar to Lemma~\ref{lem:jump-lin} holds in the case of non-linear turning rate. 
\begin{lem}\label{lem:jump-nonlin} Consider a non-linear turning rate \eqref{eq:rate} discretized by \eqref{eq:nonlinrate_disc} and \eqref{e:discr-nonlin}. 
Assume $\disc{T}_n=T_n$, $X_{T_n}=\disc{X}_{T_n}$, $\Vj_n=\disc{V}_{T_n}$ and $\disc{Y}_{T_n}=Y_{T_n}$, then 
\begin{align}
	\left|T_{n+1}-\disc{T}_{n+1}\right| \leq C \eps^2 (\delta t \theta_{n+1}^2 + \delta t ^2 \theta_{n+1}) ,
\end{align}
and consequently,
\begin{align}
	\left|X_{T_{n+1}}-\disc{X}_{\disc{T}_{n+1}}\right|\leq C \eps^3 (\delta t \theta_{n+1}^2 + \delta t ^2 \theta_{n+1}) .
\end{align}
\end{lem}
\begin{proof}
Similarly to the linear case we can show that
        \begin{equation}
          \label{eq:diff1}
          \abs{T_{n+1} - \disc{T}_{n+1}} \leq C \int_{T_n}^{T_{n+1}\wedge \disc{T}_{n+1}} \abs{\lambda(Z_t)-\disc{\lambda}\left(\disc{Z}_t,\disc{Z}_{n,\lfloor\frac{t_{n,0}-t}{\delta t} \rfloor}\right)  } \d t .
        \end{equation}	
We then estimate the righthand side of~\eqref{eq:diff1}.  This is done in two steps~:
\begin{enumerate}[Step (i)]
	\item Estimate  difference between  turning rate and  locally linearized turning rate for the discretized deviations $\left|\lambda(\disc{Z}_t)-	
			\disc{\lambda}(\disc{Z}_t,\disc{Z}_{n,k})\right|$.
	\item Estimate  difference between  turning rates for the discretized and time-continuous deviations $\abs{ \lambda({Z}_t)-\lambda(\disc{Z}_t) }$.
\end{enumerate}

Step~(i): $\left|\lambda(\disc{Z}_t)-	
		\disc{\lambda}(\disc{Z}_t,\disc{Z}_{n,k})\right|$. \\
A Taylor expansion for $t\in[t_{n,k},t_{n,k+1}]$ yields  
		\[
		\left|\lambda(\disc{Z}_t)-	
		\disc{\lambda}(\disc{Z}_t,\disc{Z}_{n,k})\right|\le C \pare{\disc{Z}_t-\disc{Z}_{n,k}}^2.
		\]
Then, remarking that,	 for $t\in[t_{n,k},t_{n,k+1}]$, we have \[
		\disc{Z}_t=\e^{-\pare{t-t_{n,k}}{\taueps}^{-1}}\disc{Z}_{n,k}+\eps\taueps\pare{\Id-\e^{\pare{-\pare{t-t_{n,k}}{\taueps}^{-1}}}} \nabla {S}( \disc{X}_{n,k} ),
		\]

	we get
	        \begin{align*}
          	\abs{\disc{Z}_t-\disc{Z}_{n,k}}&=\abs{\pare{\Id-\exp\pare{-\frac{t-t_{n,k}}{\taueps}}}\pare{\disc{Z}_{n,k}- \taueps\eps\nabla {S}( \disc{X}_{n,k} )} }, \\
                &=\abs{ m'(t-t_{n,k}) ( eps\nabla {S}( \disc{X}_{n,k} ) -\taueps^{-1} \disc{Z}_{n,k}  ) } \\
&\leq C \delta t \eps .
        \end{align*}
		In the last line of the above, we have used the estimate
	\[
	\abs{\disc{Z}_t} = \bigo(\eps^{\delta})
	\]
 as in the non discretized case,	which follows from Duhamel's formula \eqref{eq:Duhamel_num} and the assumption $\abs{\disc{Z}_{T_n}}  = \bigo (\eps^{\delta})$.
Finally, we obtain
	        \begin{align*}
	          	\int_{t_{n,k}}^{t_{n,k+1}} \left|\lambda(\disc{Z}_t)-	
		\disc{\lambda}(\disc{Z}_t,\disc{Z}_{n,k})\right| \d t \leq C \delta t^3 \eps^2,
% \\
% 	& \leq K \delta t^3 \frac{\eps^2\norm{\taueps^{-1}}^2}{r(\taueps^{-1})^2}.
	        \end{align*}
and thus:
\begin{equation*}
  \int_{T_n}^{T_{n+1}\wedge \disc{T}_{n+1}} \abs{\lambda(Z_t)-\disc{\lambda}\left(\disc{Z}_t,\disc{Z}_{n,\lfloor\frac{t_{n,0}-t}{\delta t} \rfloor}\right)  } \d t
\leq C \delta t^2 \eps^2 \theta_{n+1}
\end{equation*}
Step~(ii): $\abs{ \lambda({Z}_t)-\lambda(\disc{Z}_t) }$. \\
	Using Duhamel's formula as in the proof of Lemma~\ref{lem:jump-lin}, and using~\eqref{eq:dZ} above yields
	\begin{align*}
	  \abs{ \lambda({Z}_t)-\lambda(\disc{Z}_t) } \leq C \abs{ {Z}_t-\disc{Z}_t } \leq  C \delta t \eps^2 \theta_{n+1}.
	\end{align*}
	Finally we get, using~\eqref{eq:diff1},
	\[
	\abs{T_{n+1} - \disc{T}_{n+1}} \leq  C \eps^2(\delta t \theta_{n+1}^2 + \delta t ^2 \theta_{n+1}).
	\]
	This concludes the proof.
	
\end{proof}

\section{Conclusions and discussion\label{sec:concl}}

This paper proved rigorously, using probabilistic arguments, the convergence of a velocity-jump process for chemotaxis of bacteria with internal state to an advection-diffusion limit, and showed that the same advection-diffusion limit can be obtained as the diffusion limit of a  simpler, coarse process with direct gradient sensing.  From a numerical point of view, we introduced a time-discretized model that retains these diffusion limits.  

A direct simulation using stochastic particles presents a large statistical variance, even in the diffusive asymptotic regime, where the limiting behavior of the bacterial density is known explicitly. Consequently, it is difficult to assess accurately how the solutions of the fine-scale model differ from their advection-diffusion limit in intermediate regimes. Therefore, numerical results are deferred to the companion paper \cite{variance}, where we will use the results of this paper to construct a hybrid scheme for the simulation of the fine-scale process with ``asymptotic'' variance reduction.  The hybrid scheme simulates both the process with internal state and the simpler process with direct gradient sensing using the same random numbers, and exploits a pathwise coupling between the two trajectories. The variance reduction that is ``asymptotic'' in the sense that the variance vanishes in the diffusion limit.

\section*{Acknowledgements}

The authors thank Radek Erban, Thierry Goudon, Yannis Kevrekidis and Tony Leli\`evre for interesting discussions that eventually led to this work.  
This work was performed during a research stay of GS at SIMPAF (INRIA - Lille). GS warmly thanks the whole SIMPAF team for its hospitality. GS is a Postdoctoral Fellow of the 
Research Foundation -- Flanders. 
This work was partially supported by the Research Foundation -- Flanders
through Research Project G.0130.03 and by the Interuniversity
Attraction Poles Programme of the Belgian Science Policy Office through grant
IUAP/V/22 (GS). The scientific responsibility rests with its authors.

\bibliographystyle{plain}
%\bibliography{bibliography}
\bibliography{bib-papers}

\end{document}